\documentclass[12pt]{amsart}

\usepackage{amssymb,mathrsfs,amsmath,amsthm,color,bm}

\usepackage{enumerate}
\usepackage[centering]{geometry}
\theoremstyle{plain}
\newtheorem{thm}{Theorem}[section]
\newtheorem{defn}{Definition}[section]

\newtheorem{prop}[thm]{Proposition}

\newtheorem{lem}[thm]{Lemma}

\theoremstyle{remark}
\newtheorem{rem}{Remark}
\numberwithin{equation}{section}

\DeclareMathOperator{\hdim}{\dim_H}

\newcommand{\dif}{ \, \mathrm d}
\newcommand{\rc}{\mathcal R}

\newcommand{\N}{\mathbb N}
\newcommand{\R}{\mathbb R}

\newcommand{\lm}{\mathcal L}
\newcommand{\mc}{M^{*(d-1)}}
\newcommand{\mcl}{M_*^{(d-1)}}
\newcommand{\cu}{\mathcal U}
\newcommand{\cq}{\mathcal Q}
\newcommand{\ck}{\mathcal K}

\newcommand{\cb}{\mathcal B}

\newcommand{\cp}{\mathcal P}
\newcommand{\br}{\mathbf{r}}

\newcommand{\bx}{\mathbf{x}}
\newcommand{\bz}{\mathbf{z}}
\newcommand{\bu}{\mathbf{u}}
\newcommand{\bv}{\mathbf{v}}
\newcommand{\bt}{\mathbf{t}}

\newcommand{\pro}{(\mathbf{P1})}
\newcommand{\pros}{(\mathbf{P2})}

\begin{document}
		\title[Quantitative recurrence properties for piecewise expanding maps]{Quantitative recurrence properties for piecewise expanding maps on $ [0,1]^d $}
	\author{Yubin He}

	\address{Department of Mathematics, Shantou University, Shantou, Guangdong, 515063, China}

	\email{ybhe@stu.edu.cn}

	\author{Lingmin Liao}

	\address{School of Mathematics and Statistics, Wuhan University, Wuhan, Hubei 430072, China}

	\email{lmliao@whu.edu.cn}


	\begin{abstract}
		 Let $ T\colon[0,1]^d\to [0,1]^d $ be a piecewise expanding map with an absolutely continuous invariant measure $ \mu $. Let $ \{H_n\} $ be a sequence of hyperrectangles or hyperboloids centered at the origin. Denote by $ \rc(\{H_n\}) $  the set of points $ \bx\in[0,1]^d $ such that $ T^n\bx\in \bx+H_n $ for infinitely many $ n\in\N $, where $ \bx+H_n $ is the translation of $ H_n $. We prove that if $ \mu $ is exponential mixing and the density of $ \mu $ is sufficiently regular, then the $\mu$-measure of $ \rc(\{H_n\}) $ is zero or full according as the sum of the volumes of $ H_n $ converges or not. In the case that $ T $ is a matrix transformation, our results extend a previous work of Kirsebom, Kunde, and Persson [to appear in Ann. Sc. Norm. Super. Pisa Cl. Sci., 2023] in two aspects: by allowing the matrix to be non-integer and by allowing the `targets' $ H_n $ to be hyperrectangles or hyperboloids. We also obtain a dimension result when $ T $ is a diagonal matrix transformation.
	\end{abstract}
	\maketitle

\section{Introduction}
Let $ (X,d,T,\mu) $ be a probability measure-preserving system endowed with a compatible metric $ d $ so that $ (X,d) $ is complete and separable. Poincar\'e's recurrence theorem asserts that $ \mu $-almost all points $ x\in X $ are recurrent, i.e.
\[\liminf_{n\to\infty} d(T^nx,x)=0.\]
It is natural to ask with which rate a typical point returns close to itself.  In his pioneering paper \cite{Boshernitzan93}, Boshernitzan proved the following
quantitative recurrence result.
\begin{thm}[{\cite[Theorem 1.2]{Boshernitzan93}}]
	Let $ (X,d,T,\mu) $ be a probability measure-preserving system endowed with a metric $ d $. Assume that for some $ \alpha>0 $ the $ \alpha $-Hausdorff measure $ \mathcal H^\alpha $ is $ \sigma $-finite on $ (X,d) $. Then for $ \mu $-almost every $ x\in X $, we have
	\[\liminf_{n\to\infty} n^{1/\alpha}d(T^nx,x)<\infty.\]
	Moreover, if $ \mathcal H^\alpha(X)=0 $, then for $ \mu $-almost every $ x\in X $, we have
	\[\liminf_{n\to\infty} n^{1/\alpha}d(T^nx,x)=0.\]
\end{thm}
\noindent Later, Barreira and Saussol \cite{BarreiraSaussol01} related the recurrence rate to the local pointwise dimension.
\begin{thm}[{\cite[Theorem 1]{BarreiraSaussol01}}]
	If $ T:X\to X $ is a Borel measurable map on $ X\subset \R^d $, and $ \mu $ is a $ T $-invariant probability measure on $ X $, then for $ \mu $-almost every $ x\in X $, we have
	\[\liminf_{n\to\infty}n^{1/\alpha}d(T^nx,x)=0\quad\text{for any }\alpha>\liminf_{r\to 0}\frac{\log\mu(B(x,r))}{\log r}.\]
\end{thm}

Boshernitzan's result can be reformulated as: for $ \mu $-almost every $ x $, there is a constant $ c(x)>0 $ such that
\[d(T^nx,x)<c(x)n^{-1/\alpha}\quad\text{ for infinitely many }n\in\N.\]
This then leads us to study the recurrence set
\begin{equation}\label{eq:recurrence set ball}
	\rc(\{r_n\}):=\{x\in X: d(T^nx,x)<r_n\text{ for infinitely many } n\in\N\},
\end{equation}
where $ \{r_n\} $ is a sequence of non-negative real numbers. One asks how large the size of $ \rc(\{r_n\}) $ is in the sense of measure and in the sense of Hausdorff dimension. A simple but important observation, which will be central to most of what follows, is that $ \rc(\{r_n\}) $ and other related sets that we are interested in can be expressed as $ \limsup $ sets, i.e.
\begin{equation*}\label{eq:expressed as limsup set}
	\rc(\{r_n\})=\bigcap_{N=1}\bigcup_{n\ge N}\{x\in X: d(T^nx,x)<r_n\}.
\end{equation*}

For the size of $ \rc(\{r_n\}) $ in measure, the mainstream has been to obtain the $\mu$-measure of $\rc(\{r_n\}) $, where $ \mu $ is an (absolutely continuous) invariant measure. Specifically, one may expect that $\mu(\rc(\{r_n\})) $ obeys a zero-one law according as the convergence or divergence of the series $ \sum_{n=1}^{\infty}r_n^{\hdim X} $, where $ \hdim $ stands for the Hausdorff dimension. Such a zero-one law, also referred to as dynamical Borel--Cantelli lemma, has been verified in various setups. For example, the case where $ X $ is a homogeneous self-similar set satisfying strong separation condition was investigated by Chang, Wu and Wu \cite{CWW19}. Their result was subsequently generalized by Baker and Farmer \cite{BakerFarmer21} to finite conformal iterated function systems satisfying the open set condition, and further, by Hussain, Li, Simmons and Wang \cite{HLSW22} to more general conformal and expanding dynamical systems. Later on, the results presented in \cite{HLSW22} were refined by Kleinbock and Zheng \cite{KleinbockZheng22}. It should be noticed that all the results mentioned above are only applicable to conformal dynamical systems. The non-conformal case where $ T $ is an expanding integer matrix transformation was recently studied by Kirsebom, Kunde and Persson \cite{KKP21}.  On the other hand, Allen, Baker and B\'ar\'any \cite{ABB22} studied the recurrence rates for shifts of finite type that hold $\mu$-almost surely with respect to a Gibbs measure $ \mu $, and presented a nearly complete description of the $\mu$-measure of the recurrence set.

For the size of $ \rc(\{r_n\}) $ in Hausdorff dimension, it was first studied by Tan and Wang \cite{TanWang11} when the underlying dynamical system is the $\beta$-transformation. Since their initial work on the Hausdorff dimension, many other cases have also been addressed. We can refer to \cite{SeWa15} for conformal iterated function systems,  to \cite{YuLi23} for expanding Markov maps,  to \cite{WuYu22} for certain self-affine maps, to \cite{HuPe23} for the cat map, and to \cite{YuWa23} for higher dimensional $\beta$-transformation.

The studies mentioned above have a deep connection with the {\em shrinking target problems}. Let $ \{S_n\} $ be a sequence of subsets of $ X $. The shrinking target problems concern the measure and the Hausdorff dimension of the set
\begin{equation}\label{eq:shrinking targets}
	\{x\in X: T^nx\in S_n\text{ for infinitely many } n\in\N\}=\bigcap_{N=1}\bigcup_{n\ge N}T^{-n}S_n.
\end{equation}
For more details on shrinking target problems we refer the reader to \cite{AlBa21,BarRa18,ChKl01,FMP07,GaKim07,HiVe95,HiVe99,LLVZ22,LWWW14,Lopez20,WaZh21} and references therein. Of particular interest to us is the recent work by Li, Liao, Velani and Zorin \cite{LLVZ22} in which the shrinking target problems for the matrix transformations of tori were fully studied. In their paper \cite{LLVZ22}, the authors allowed the shrinking targets $ S_n $ to be hyperrectangles or hyperboloids and gave some general conditions for a real, non-singular matrix transformation to guarantee that the $ d $-dimensional Lebesgue measure of the shrinking target set obeys a zero-one law. In the case that $ S_n $ are hyperrectangles with sides parallel to the axes, they also determined the Hausdorff dimension of the corresponding shrinking targets set for diagonal matrix transformations.

For the quantitative recurrence theory, little is known for the matrix transformation, except a recent work by Kirsebom, Kunde and Persson \cite{KKP21} in which a criterion for determining the measure of the recurrence set was provided. However, the result of \cite{KKP21} is valid only for expanding integer matrix transformation but not for expanding real matrix transformation. The main purpose of this paper is to investigate the more general setup in which $ T $ is a piecewise expanding map (including expanding real matrix transformation) and the `targets' are either hyperrectangles or hyperboloids.

We start with introducing some necessary notations and definitions. In the current work, we take $ X $ to be the unit cube $ [0,1]^d $ endowed with the maximum norm $ |\cdot| $, i.e.\,for any $ \bx=(x_1,\dots, x_d)\in [0,1]^d $, $ |\bx|=\max\{|x_1|,\dots,|x_d|\} $. Throughout this paper, we adhere to the following notation. The closure, boundary, $\varepsilon$-neighbourhood, diameter, and cardinal number of $ A $ will be denoted by $ \overline A $, $ \partial A $, $ A(\varepsilon) $, $ |A| $ and $  \#  A $, respectively. The notations $ \lm^d $ and $ \lm^{d-1} $ stand for the Lebesgue measures of dimensions $ d $ and $ d-1 $, respectively. The $ (d-1) $-dimensional upper and lower Minkowski contents are defined, respectively as
\[\mc(A):=\limsup_{\epsilon\to 0^+}\frac{\lm^d(A(\epsilon))}{\epsilon}\quad\text{and}\quad\mcl(A):=\liminf_{\epsilon\to 0^+}\frac{\lm^d(A(\epsilon))}{\epsilon}.\]
If the $ (d-1) $-dimensional upper and lower Minkowski contents are equal, then their common value is called the $ (d-1) $-dimensional Minkowski content of $ A $ and is denoted by $ M^{d-1}(A) $. Note that if $ A $ is a closed $ (d-1) $-rectifiable subset of $ \R^d $, i.e.\,the image of a
bounded set from $ \R^{d-1} $ under a Lipschitz function, then $ M^{d-1}(A) $ exists and equals to $ \lm^{d-1}(A) $ multiplied by a constant. For further details, see \cite{Federer,Krantz} and references within.

Let us now specify the class of measure-preserving systems $ ([0,1]^d,T,\mu) $ which we can treat by our techniques.

\begin{defn}[Piecewise expanding map]\label{d:C1 map}
	We say that $ T\colon[0,1]^d\to [0,1]^d $ is a piecewise expanding map if there is a finite family $ (U_i)_{i=1}^P $ of pairwise disjoint connected open subsets in $ [0,1]^d $ with $ \bigcup_{i=1}^P\overline{U_i}=[0,1]^d $ such that the following statements hold.
	\begin{enumerate}[\upshape(i)]
		\item The map $ T $ is injective and can be extended to a $ C^1 $ map on each $ \overline U_i $. Moreover, there exists a constant $ L>1 $ such that
		\begin{equation}\label{eq:expanding constant}
			\inf_i\inf_{\bx\in U_i} \|D_\bx T\|\ge L,
		\end{equation}
		where $ D_\bx T $ is the differential of $ T $ at $ \bx $ determined by
		\[\lim_{|\bz|\to 0}\frac{|T(\bx+\bz)-T(\bx)-(D_{\bx}T)(\bz)|_2}{|\bz|_2}=0\]
		and $ \|D_\bx T\|:=\sup_{\bz\in\R^d}|(D_\bx T)(\bz)|/|\bz| $. Here $ |\bz|_2=\sqrt{z_1^2+\cdots+z_d^2} $.
	\item The boundary of $ U_i $ is included in a $ C^1 $ piecewise embedded compact submanifold of codimension one. In particular, there exists a constant $ K $ such that
	\[\max_{1\le i\le P}\mc(\partial U_i)\le K.\]
	\end{enumerate}
\end{defn}

\begin{rem}\label{r:1}
	Define
	\[\mathcal F_n=\{U_{i_0}\cap T^{-1}U_{i_1}\cap\cdots\cap T^{-(n-1)}U_{i_{n-1}}:1\le i_0,i_1,\dots,i_{n-1}\le P\}.\]
	In Definition \ref{d:C1 map}, the smoothness assumptions on the map $ T $ and the boundaries of $ (U_i)_{i=1}^P $ are to facilitate the estimation of $ \mc(\partial J_n) $ for all $ J_n\in\mathcal F_n $ (see Lemma \ref{l:boundary estimation}).
\end{rem}
Throughout, we will always assume that $ \mu $ is an absolutely continuous (with respect to the Lebesgue measure $ \mathcal L^d $) $ T $-invariant probability measure. Under this assumption, we see that
\[\mu\bigg(\bigcup_{1\le i\le P}U_i \bigg)=1.\]
We will need the following definition taken from \cite{LLVZ22}, which is a variation of \cite[Theorem 6.1]{Saussol00}.
\begin{defn}[Exponential mixing]\label{d:exponential mixing}
	For any collection $ \mathcal C $ of measurable subsets $ F $ of $ [0,1]^d $ satisfying
	\[\sup_{F\in\mathcal C}\mc(\partial F)<\infty,\]
	there exist constants $ c>0 $ and $ \tau>0 $ such that for any $ F, G\in\mathcal C $, we have
	\begin{equation}\label{eq:expoential decay definition}
		|\mu(F\cap T^{-n}G)-\mu(F)\mu(G)|\le c\mu(G)e^{-\tau n}.
	\end{equation}
\end{defn}

We consider two special families of sets that related to the setups of the classical theory of Diophantine approximation. For any $ \bx\in [0,1]^d $, $ \br\in(\R_{\ge 0})^d $ and $ \delta>0 $, let
\[R(\bx,\br):=\prod_{i=1}^{d}B(x_i,r_i)\]
and
\[H(\bx,\delta):=\bx+\{\bz\in [-1,1]^d: |z_1\cdots z_d|<\delta\}.\]
The sets $ R(\bx,\br) $ and $ H(\bx,\delta) $ are usually called targets. Clearly, $ R(\bx,\br) $ is a hyperrectangle with sides parallel to the axes and $ H(\bx,\delta) $ is a hyperboloid. Let $ \{\br_n\} $ be a sequence of vectors with $ \br_n=(r_{n,1},\dots,r_{n,d})\in(\R_{\ge 0})^d $ and
\begin{equation}\label{eq:rn to 0}
	\lim_{n\to\infty}|\br_n|=0.
\end{equation}
Define
\begin{equation}\label{eq:rectangle}
	\rc(\{\br_n\})=\bigcap_{N=1}\bigcup_{n\ge N}\{\bx\in [0,1]^d:T^n\bx\in R(\bx,\br_n)\}.
\end{equation}
If all the entries of $ \br_n $ coincide, that is $ r_{n,1}=\cdots=r_{n,d}=r_n $, then $ \rc(\{\br_n\}) $ is nothing but $\rc(\{r_n\}) $ defined in \eqref{eq:recurrence set ball}.
Let $ \{\delta_n\} $ be a sequence of non-negative real numbers such that
\begin{equation}\label{eq:deltan to 0}
	\lim_{n\to\infty}\delta_n=0.
\end{equation}
Define
\begin{equation}\label{eq:hyperbolic}
	\rc^\times (\{\delta_n\})=\bigcap_{N=1}\bigcup_{n\ge N}\{\bx\in[0,1]^d:T^n\bx\in H(\bx,\delta_n)\}.
\end{equation}
We remark that we do not assume that the sequences $ \{\br_n\} $ and $ \{\delta_n\} $ are non-increasing.
The definitions of $ \rc(\{\br_n\}) $ and $ \rc^\times (\{\delta_n\}) $ are motivated respectively by the weighted and multiplicative theories of classical Diophantine approximation. To illustrate this, let $ T $ be a $ d\times d $ non-singular matrix with real coefficients. Then, $ T $ determines a self-map on $ [0,1)^d $; namely, it sends $ \bx\in[0,1)^d $ to $ T\bx\mod 1 $. Here and below, when $ T $ is a matrix, the notation $ T $ will denote both the matrix and the transformation. Suppose that $ T=\mathrm{diag}(\beta_1,\dots,\beta_d) $ is a diagonal matrix with $ |\beta_i|>1 $. Then, for any $ \bx\in [0,1]^d $, we have
\begin{equation}\label{eq:recurrence rectangle}
	T^n\bx\in R(\bx,\br_n)\quad\Longleftrightarrow\quad |T_{\beta_i}x_i-x_i|<r_{n,i}\quad\text{for all }1\le i\le d,
\end{equation}
and
\[T^n\bx\in H(\bx, \delta_n)\quad\Longleftrightarrow\quad \prod_{i=1}^{d}|T_{\beta_i}x_i-x_i|<\delta_n, \]
where $ T_{\beta_i} $ is the standard $ \beta $-transformation with $ \beta=\beta_i $ given by
\[T_\beta x=\beta x-\lfloor\beta x\rfloor.\]
Here $ \lfloor\cdot\rfloor $ denotes the integer part of a real number.

Our main results are stated below.
\begin{thm}\label{t:main}
	Let $ T\colon[0,1]^d\to[0,1]^d $ be a piecewise expanding map. Suppose that $ \mu $ is exponential mixing and the density $ h $ of $ \mu $ belongs to $ L^q(\lm^d) $ for some $ q > 1 $. Then,
\[\mu(\rc(\{\br_n\}))=\begin{cases}
	0\quad&\text{if }\sum_{n=1}^{\infty}r_{n,1}\cdots r_{n,d}<\infty, \\
	1\quad&\text{if }\sum_{n=1}^{\infty}r_{n,1}\cdots r_{n,d}=\infty.
\end{cases}\]
\end{thm}
\begin{thm}\label{t:main s}
	Let $ T\colon[0,1]^d\to[0,1]^d $ be a piecewise expanding map. Suppose that $ \mu $ is exponential mixing, and that there exists an open set $ V $ with $ \mu(V)=1 $ such that the density $ h $ of $ \mu $, when restricted on $ V $, is bounded from above by $ \mathfrak{c}\ge 1 $ and bounded from below by $ \mathfrak{c}^{-1} $. Then,
	\[\mu(\rc^\times(\{\delta_n\}))=\begin{cases}
		0\quad&\text{if }\sum_{n=1}^{\infty}\delta_n(-\log \delta_n)^{d-1}<\infty, \\
		1\quad&\text{if }\sum_{n=1}^{\infty}\delta_n(-\log \delta_n)^{d-1}=\infty.
	\end{cases}\]
\end{thm}
\begin{rem}
	As mentioned in \cite[Remark 1.8]{HLSW22}, for the shrinking target problems, the mixing property and the invariance property of $ \mu $ can be applied directly to verify the quasi-independence of the events $ \{T^{-n}S_n\} $ in \eqref{eq:shrinking targets}. While for the recurrence theory, the events in \eqref{eq:rectangle} and \eqref{eq:hyperbolic} cannot be expressed as the $ T $-inverse images of some sets, which makes the proofs of zero-one laws more involved. We will apply some of the ideas from \cite{HLSW22,KKP21}, but the proofs are quite different due to the new setup.
\end{rem}
\begin{rem}
	The assumptions in Theorem \ref{t:main s} are stronger than those in Theorem \ref{t:main}. According to our method, the shapes of targets do play a role in the proof. When the targets are hyperrectangles with sides parallel to the axes, the Zygmund differentiation theorem (see Theorem \ref{t:differentiation theorem}) is applicable, and Theorem \ref{t:main} can be established under a weaker assumption on $ \mu $. However, for the hyperboloid setup, to the best of our knowledge, there does not exist an analogue differentiation theorem. Therefore, we require a stronger condition on $ \mu $ to establish Theorem \ref{t:main s}.
\end{rem}
\begin{rem}
	Our method in proving Theorem \ref{t:main s} also allows us to study some more general settings.  More precisely, the targets $ \{H_n\} $ can be chosen as a sequence of parallelepipeds or ellipsoids centered at the origin such that $ \lim_{n\to\infty} |H_n|=0 $. We stress that when $ H_n $ is a hyperrectangle, the sides of $ H_n $ are not required to be parallel to the axes.  By making some obvious modifications to the proof of Theorem \ref{t:main s}, one can show that the $ \mu $-measure of the set of points $ \bx $ satisfying $ T^n\bx\in \bx+H_n $ for infinitely many $ n\in\N $, is zero or one according as the convergence or divergence of the volume sum $ \sum_{n\ge 1}\lm^{d}(H_n) $. However, the proofs will be a bit lengthy without yielding interesting applications, and thus will not be presented in this paper.
\end{rem}
\begin{rem}
	Let us focus on the one dimensional case for a moment. In this case, the two sets $ \rc(\{\br_n\}) $ and $ \rc^\times (\{\delta_n\}) $ are identical, and we write $ \rc(\{r_n\}) $ in place of $ \rc(\{\br_n\}) $ or $ \rc^\times (\{\delta_n\}) $. Since the geometry of $ [0,1] $ is relatively simple, Theorem \ref{t:main} still holds if the condition stated in Definition \ref{d:C1 map} (i) is relaxed to that $ T $ is strictly monotonic and continuous on each $ U_i $. The detailed discussion will be given in Lemma \ref{l:boundary estimation} and Remark \ref{r: relax condition on T}. In  \cite{KKP21}, Kirsebom, Kunde and Persson proved that under some conditions similar to Theorem \ref{t:main} but without assuming that $ T $ is piecewise expanding,
	\[\sum_{n\ge 1}\int\mu(B(x,r_n))\,d\mu(x)<+\infty\quad\Longrightarrow\quad \mu(\rc(\{r_n\}))=0.\]
	They also posed a question of whether the complementary divergence statement
	holds. Our result indicates that this is not always the case. If one adds an extra assumption to \cite[Theorem C]{KKP21} that $ T $ is a piecewise monotone continuous function, then the $\mu$-measure of $ \rc(\{r_n\}) $ is determined by the sum $ \sum_{n\ge 1}r_n $ rather than $ \sum_{n\ge 1}\int\mu(B(x,r_n))\,d\mu(x) $. It is easily verified that
	\[\sum_{n\ge 1}\int\mu(B(x,r_n))\,d\mu(x)<+\infty\quad\Longrightarrow\quad \sum_{n\ge 1}r_n<\infty.\]
	However, the reverse implication is unclear.
\end{rem}

As an immediate consequence of Theorem \ref{t:main}, we can extend the result of \cite[Theorem D]{KKP21} to some more general settings.

\begin{thm}\label{t:sub}
	Let $ T $ be a $ d\times d $ real matrix with the modulus of all eigenvalues strictly larger than $ 1 $. Suppose that $ T $ satisfies one of the following conditions.
	\begin{enumerate}
		\item All eigenvalues of $ T $ are of modulus strictly larger than $ 1+\sqrt d $.
		\item $ T $ is diagonal.
		\item $ T $ is an integer matrix.
	\end{enumerate}
	Then,
	\[\mu(\rc(\{\br_n\}))=\begin{cases}
		0\quad&\text{if }\sum_{n=1}^{\infty}r_{n,1}\cdots r_{n,d}<\infty, \\
		1\quad&\text{if }\sum_{n=1}^{\infty}r_{n,1}\cdots r_{n,d}=\infty.
	\end{cases}\]
\end{thm}
\begin{rem}
	The conditions stated in items (1)--(3) are consistent with those stated in \cite[Theorems 3--5]{LLVZ22}, respectively. As pointed out in \cite[Proposition 1]{LLVZ22}, Saussol's result \cite[Theorem 6.1]{Saussol00} implies that the absolutely continuous invariant measure $ \mu $ is exponential mixing if $ T $ satisfies one of the conditions specified in Theorem \ref{t:sub}. Moreover, Saussal \cite[Proposition 3.4 and Theorem 5.1 (ii)]{Saussol00} proved that the density of $ \mu $ is bounded from above. Hence, Theorem \ref{t:main} applies.
\end{rem}

\begin{thm}\label{t:sub m}
	Let $ T $ be a $ d\times d $ real matrix with the modulus of all eigenvalues strictly larger than $ 1 $. Suppose that $ T $ satisfies one of the following conditions.
	\begin{enumerate}
		\item $ T $ is diagonal with all eigenvalues belonging to $ (-\infty, -(\sqrt 5+1)/2]\cup (1,+\infty) $,
		\item $ T $ is an integer matrix.
	\end{enumerate}
	Then,
	\[\mu(\rc^\times(\{\delta_n\}))=\begin{cases}
		0\quad&\text{if }\sum_{n=1}^{\infty}\delta_n(-\log \delta_n)^{d-1}<\infty, \\
		1\quad&\text{if }\sum_{n=1}^{\infty}\delta_n(-\log \delta_n)^{d-1}=\infty.
	\end{cases}\]
\end{thm}
\begin{rem}
	If $ T $ satisfies one of the conditions in Theorem \ref{t:sub m}, then the density of the absolutely continuous invariant measure $ \mu $ is bounded from above and bounded away from zero. See \cite[Proposition 2]{LLVZ22} for more details. Hence, we can apply Theorem \ref{t:main s}.
\end{rem}

We also address the Hausdorff dimension of the recurrence set in the case that $ T $ is a diagonal matrix with the modulus of  all eigenvalues $ |\beta_1|,\dots,|\beta_d| $ strictly larger than 1. For $ 1\le i\le d $, let $ \psi_i:\R_{\ge 0}\to\R_{\ge 0} $ be a positive and non-increasing function. For convenience, let $ \Psi:=(\psi_1,\dots,\psi_d) $ and for $ n\in\N $ let $ \Psi(n):=(\psi_1(n),\dots,\psi_d(n)) $. Define
\[\rc(\Psi)=\{\bx\in[0,1]^d:|T^n_{\beta_i}x_i-x_i|<\psi_i(n)\ (1\le i\le d)\text{ for infinitely many }n\in\N\}.\]
By \eqref{eq:recurrence rectangle}, it is easily seen that $ \bx\in \rc(\Psi) $ if and only if $ T^n\bx\in R(\bx, \Psi(n)) $ for infinitely many $ n\in\N $. We stress that the targets $ R(\bx, \Psi(n)) $ are hyperrectangles whose sides are parallel to the axes.

It turns out that the Hausdorff dimension of $ \rc(\Psi) $ depends on the set $ \cu(\Psi) $ of accumulation points $ \bt=(t_1,\dots,t_d) $ of the sequence $ \big\{\big(-\frac{\log\psi_1(n)}{n},\dots,-\frac{\log\psi_d(n)}{n}\big)\big\}_{n\ge 1} $.
\begin{thm}\label{t:dimension}
	Let $ T $ be a real matrix transformation of $ [0,1]^d $. Suppose that $ T $ is diagonal with all eigenvalues $ \beta_1,\dots,\beta_d $ of modulus strictly larger than 1. For $ 1\le i\le d $, let $ \psi_i:\R_{\ge 0}\to\R_{\ge 0} $ be a positive and non-increasing function. Assume that $ \cu(\Psi) $ is bounded. Then,
	\begin{equation}\label{eq:dimrp}
		\hdim \rc(\Psi)=\sup_{\bt\in\cu(\Psi)}\min_{1\le i\le d}\theta_i(\bt),
	\end{equation}
	where
	\[\theta_i(\bt):=\sum_{k\in\ck_1(i)}1+\sum_{k\in\ck_2(i)}\Big(1-\frac{t_k}{\log|\beta_i|+t_i}\Big)+\sum_{k\in\ck_3(i)}\frac{\log|\beta_k|}{\log|\beta_i|+t_i}\]
	and, in turn
	\[\ck_1(i):=\{1\le k\le d:\log|\beta_k|>\log |\beta_i|+t_i\},\]
	\[\ck_2(i):=\{1\le k\le d:\log|\beta_k|+t_k\le \log|\beta_i|+t_i\}\]
	and
	\[\ck_3(i):=\{1,\dots,d\}\setminus(\ck_1(i)\cup\ck_2(i)).\]
\end{thm}
\begin{rem}
	The monotonicity of $ \psi_i $ ($ 1\le i\le d $) can be removed if each $ \beta_i $ is positive. For details on removing this condition, we refer to \cite[Proof of Theorem 12]{LLVZ22}.
\end{rem}
\begin{rem}\label{r:comp}
	Under the setting of Theorem \ref{t:dimension}, Li, Liao, Velani and Zorin \cite[\S 5.2]{LLVZ22} proved that the Hausdorff dimension of shrinking target set
	\[W(\Psi):=\{\bx\in[0,1]^d:|T^n_{\beta_i}x_i-y_i|<\psi_i(n)\ (1\le i\le d)\text{ for infinitely many }n\in\N\}\]
	is greater than or equal to the right of \eqref{eq:dimrp}. When $ \beta_i>1 $ for $ 1\le i\le d $, they further showed that the upper bound of $ \hdim W(\Psi) $ coincides with this lower bound by utilizing the result of Bugeaud and Wang \cite[Theorem 1.2]{BuWa2014}. It is worth noting that the  method employed in establishing the upper bound of $ \hdim \rc(\Psi) $ does not rely on Bugeaud and Wang's result and can also be applied to obtain the desired upper bound for $ \hdim W(\Psi) $. Consequently, under the setting of Theorem \ref{t:dimension}, we have
	\[\hdim W(\Psi)=\sup_{\bt\in\cu(\Psi)}\min_{1\le i\le d}\theta_i(\bt),\]
	which affirmatively answers a question raised in \cite[Claim 1]{LLVZ22}.

	\begin{rem}
		Note that the sets $ W(\Psi) $ and $ \rc(\Psi) $ share the same dimensional formulae. So, one would like to treat the dimensions of these two sets in a unified way by considering the following set
		\[W(\Psi,f):=\{\bx\in[0,1]^d:T^n\bx\in R(f(\bx), \Psi(n))\text{ for infinitely many }n\in\N\},\]
		where $ f $ is a Lipschitz function and $ T $ is a diagonal matrix given in Theorem \ref{t:dimension}. Yuan and Wang \cite[Theorem 1.1]{YuWa23} proved that under the condition $ \beta_i>1 $ for $ 1\le i\le d $, $ \hdim W(\Psi,f) $ also equals to the right of $ \eqref{eq:dimrp} $. However, when $ \beta_i<-1 $ for some $ 1\le i\le d $, the Hausdorff dimension of $ W(\Psi,f) $ remains unknown except for two special cases $ f(\bx)=\bx $ and $ f(\bx)\equiv \textbf{y} $ discussed above. We can obtain the upper bound of $ \hdim W(\Psi, f) $ by applying the same techniques in our present paper. But our method using the approximation of Markov subsystems fails in obtaining the lower bound.
	\end{rem}
\end{rem}

Our paper is organized as follows. The convergence part of Theorem \ref{t:main} is given in Section \ref{s:convergence}. The divergence part will be proved in Section \ref{s:divergence}. More precisely, we first introduce a sequence of auxiliary sets $ \hat E_n $ and estimate their measures in Section \ref{s:auxiliary sets}, and then estimate the correlations of these sets in Section \ref{s:correlation}. This combined with a technical argument allows us to conclude the divergence part of Theorem \ref{t:main} in Section \ref{s:full measure}. In Section \ref{s:hyperbola}, we adopt a similar but more direct approach compared to the proof of Theorem \ref{t:main} to establish Theorem \ref{t:main s}. The last section is reserved for determining the Hausdorff dimension of $ \rc(\Psi) $.

\section{Proof of Theorem \ref{t:main}}
Throughout this section, we will always assume that $ \mu $ is exponential mixing and the density $ h $ of $ \mu $ belongs to $ L^q(\lm^d) $ for some $ q > 1 $. Besides, we fix the collection $ \mathcal C_1 $ of subsets $ F\subset [0,1]^d $ satisfying the bounded property
\[\pro: \qquad\sup_{F\in\mathcal C_1}\mc(\partial F)<4d+\frac{K}{1-L^{-(d-1)}},\]
where $ L $ and $ K $ are constants given in Definition \ref{d:C1 map}.
For any hyperrectangle $ R\subset [0,1]^d $, its boundary consists of $ 2d $ hyperrectangles of dimension $ d-1 $, each of which has a $ (d-1) $-dimensional Lebesgue measure less than $ 1 $. Thus, $ \mc(\partial R)\le 2d $ and the family of hyperrectangles satisfies the bounded property $ \pro $.

\subsection{Convergence part}\label{s:convergence}
This subsection is devoted to proving the following proposition that applies to the convergence part of Theorem \ref{t:main}.

\begin{prop}\label{p:converges implies zero measure}
	If $ \sum_{n=1}^{\infty}r_{n,1}\cdots r_{n,d}<\infty $ ,then
	\[\mu(\rc(\{\br_n\}))=0.\]
\end{prop}
Note that $ \rc(\{\br_n\})=\limsup E_n $, where
\[E_n:=\{\bx\in [0,1]^d:T^n\bx\in R(\bx,\br_n)\}.\]
We follow the idea from \cite[Lemma 2.2]{HLSW22} that when considering $ E_n $ locally, the set $ E_n $ can be approximated by a hyperrectangle intersecting with the $ n $-th inverse of another hyperrectangle.

\begin{lem}\label{l:simple lemma}
	Let $ R(\bx,\br) $ be a hyperrectangle with center $ \bx=(x_1,\dots,x_d)\in [0,1]^d $ and $ \br=(r_1,\dots,r_d)\in(\R_{\ge 0})^d $. Then, for any subset $ F $ of $ R(\bx,\br) $,
	\[F\cap T^{-n}R(\bx,\br_n-\br)\subset F\cap E_n\subset F\cap T^{-n} R(\bx,\br_n+\br).\]
\end{lem}
\begin{proof}
	Fix a point $ \bz=(z_1,\dots,z_d)\in F\cap E_n $. Then, $ \bz\in R(\bx,\br) $ and $ T^n\bz\in R(\bz,\br_n) $. Write $ T^n\bz=(z_1',\dots, z_d') $. It follows that for any $ 1\le i\le d $,
	\[|z_i-x_i|<r_i\quad\text{and}\quad| z_i'-z_i|<r_{n,i}.\]
	By using the triangle inequality, one has
	\[|z_i'-x_i|<|z_i'-z_i|+|z_i+x_i|<r_{n,i}+r_i\quad\text{for }1\le i\le d.\]
	This implies $ T^n\bz\in R(\bx,\br_n+\br) $. Therefore,
	\[F\cap E_n\subset F\cap T^{-n} R(\bx,\br_n+\br).\]

	The first inclusion follows similarly.
\end{proof}

Since the density $ h $ of $ \mu $ may be unbounded, the $\mu$-measures of hyperrectangles depend on their location. A differentiation theorem discovered by Zygmund tells us that if $ h $ belongs to $ L^q(\lm^d) $ for some $ q> 1 $, then for $\mu$-almost all points $ \bx $, the $ \mu $-measure of a small hyperrectangle $ R $ with center $ \bx $ and sides parallel to the axes is roughly equal to the volume of $ R $ multiplied by $ h(\bx) $. The proof can be found in \cite[Theorem 2.29]{Hajlasz}.

\begin{thm}[Zygmund differentiation theorem]\label{t:differentiation theorem}
	Let $ \{\br_n\} $ be a sequence of positive vectors with $ \lim_{n\to\infty}|\br_n|\to 0 $. If $ f\in L^q(\lm^d) $ for some $ q>1 $, then
	\[\lim_{n\to\infty}\frac{\int_{R(\bx,\br_n)} f(\bz)\dif\lm^d(\bz)}{\lm^d(R(\bx,\br_n))}=f(\bx)\qquad\text{for }\lm^d\text{-a.e.\,}\bx.\]
\end{thm}
\begin{rem}
Zygmund differentiation theorem works only for hyperrectangles with sides parallel to the axes, but not for other cases in general. If one allows the hyperrectangles to have different rotations then Theorem \ref{t:differentiation theorem} no longer holds.
\end{rem}

Now, let us go back to the proof of Proposition \ref{p:converges implies zero measure}. Since $ h $ belongs to $ L^q(\lm^d) $ for some $ q>1 $, by Theorem \ref{t:differentiation theorem},
\[0\le \lim_{n\to \infty}\frac{\mu(R(\bx,2\br_n))}{\lm^d(R(\bx,2\br_n))}=\lim_{n\to \infty}\frac{\mu(R(\bx,2\br_n))}{4^dr_{n,1}\cdots r_{n,d}}=h(\bx)<\infty\quad\text{for }\mu\text{-a.e.\,}\bx\in [0,1]^d.\]
This implies that
\begin{equation}\label{eq:decomposition}
	\mu\bigg(\bigcup_{k=1}^\infty\bigcup_{l=1}^\infty  Z(k,l)\bigg)=1,
\end{equation}
where
\[ Z(k,l):=\{\bx\in [0,1]^d: \mu(B(\bx,2\br_n))\le kr_{n,1}\cdots r_{n,d}\text{ for all }n\ge l\}.\]

\begin{proof}[Proof of Proposition \ref{p:converges implies zero measure}]
	By \eqref{eq:decomposition}, we fix one $ Z(k,l) $ with $ \mu(Z(k,l))>0 $. Let $ n\in\N $ with $ n\ge l $. We will inductively define a finite family $ \{R(\bx_i,\br_n):\bx_i\in  Z(k,l)\}_{i\in\mathcal I} $ of  hyperrectangles such that the union of these hyperrectangles covers $ Z(k,l) $, and that the collection of `$ 1/2 $-scaled up' hyperrectangles $ \{R(\bx_i,\br_n/2)\}_{i\in\mathcal I} $ is pairwise disjoint.

	Choose $ \bx_1\in Z(k,l) $ and let $ R(\bx_1,\br_n) $ be a hyperrectangle centered at $ \bx_1 $. Inductively, assume that $ R(\bx_1,\br_n),\dots, R(\bx_j,\br_n) $ have been defined for some $ j\ge 1 $. If the union $ \cup_{i\le j}R(\bx_i,\br_n) $ does not cover $ Z(k,l) $, let $ \bx_{j+1}\in Z(k,l)\setminus\cup_{i\le j}R(\bx_i,\br_n)  $. Otherwise, we set $ \mathcal I=\{1,\dots,j\} $ and terminate the inductive definition. It remains to show that $ \{R(\bx_i,\br_n/2)\}_{i\in\mathcal I} $ is pairwise disjoint. For any $ i $ and $ j $ with $ i<j $, by definition we have $ \bx_{j}\notin R(\bx_i,\br_n) $. Therefore, $ R(\bx_i,\br_n/2)\cap R(\bx_j,\br_n/2)=\emptyset $.

	The disjointness of $ \{R(\bx_i,\br_n/2)\}_{i\in\mathcal I} $ implies that
	\begin{equation}\label{eq:cardinality I}
		 \#  \mathcal I\le (r_{n,1}\cdots r_{n,d})^{-1}.
	\end{equation}
	By Lemma \ref{l:simple lemma}, we have
	\[ Z(k,l)\cap E_n\subset \bigcup_{i\in\mathcal I}R(\bx_i,\br_n)\cap E_n\subset\bigcup _{i\in\mathcal I}R(\bx_i,\br_n)\cap T^{-n}R(\bx_i,2\br_n).\]
	Using the exponential mixing property of $ \mu $, \eqref{eq:cardinality I} and the definition of $  Z(k,l) $, we have
	\[\begin{split}
		\mu( Z(k,l)\cap E_n)&\le \sum_{i\in \mathcal I}\mu(R(\bx_i,\br_n)\cap T^{-n}R(\bx_i,2\br_n))\\
		&\le \sum_{i\in \mathcal I}\big(\mu(R(\bx_i,\br_n))+ce^{-\tau n}\big)\mu(R(\bx_i,2\br_n))\\
		&\le (r_{n,1}\cdots r_{n,d})^{-1}\cdot(kr_{n,1}\cdots r_{n,d}+ce^{-\tau n})\cdot kr_{n,1}\cdots r_{n,d}\\
		&=k(kr_{n,1}\cdots r_{n,d}+ce^{-\tau n}).
	\end{split}\]

Since the sum $ \sum_{n\ge 1}r_{n,1}\cdots r_{n,d} $ converges, we get
\[\sum_{n=1}^\infty\mu( Z(k,l)\cap E_n)<\infty.\]
By Borel--Cantelli lemma, $ \mu( Z(k,l)\cap \rc(\{\br_n\}))=0 $. Finally, it follows from \eqref{eq:decomposition} that
\[\mu(\rc(\{\br_n\}))=0.\qedhere\]
\end{proof}

\subsection{Divergence part}\label{s:divergence}

In the sequel, we assume that $ \sum_{n\ge 1} r_{n,1}\cdots r_{n,d}=\infty $. Without loss of generality, we further assume that for all $ n\in\N $, either $ r_{n,i} $ is greater than $ n^{-2} $ for all $ 1\le i\le d $, or equal to $ 0 $ for all $ 1\le i\le d $. We will prove that the set $ \rc(\{\br_n\}) $ has full $\mu$-measure under these assumptions. In fact, note that \begin{equation}\label{eq:subset}
	\limsup_{n\colon \forall  i,\, r_{n,i}> n^{-2}}E_n\subset \limsup_{n\to\infty}E_n
\end{equation}
and
\[\sum_{n\colon \forall  i,\, r_{n,i}> n^{-2}}r_{n,1}\cdots r_{n,d}=\infty\quad\Longleftrightarrow\quad\sum_{n=1}^{\infty}r_{n,1}\cdots r_{n,d}=\infty.\]
Then, the original set $ \rc(\{\br_n\}) $ will be of full measure provided that the set in the left of \eqref{eq:subset} has full measure.
\subsubsection{A sequence of auxiliary sets $ \hat E_n $}\label{s:auxiliary sets}
\

Since the density $h$ of $\mu$ may be unbounded, instead of the sets $ E_n $, it is more suitable to study the following auxiliary sets $ \hat E_n $. The idea of the constructions of $ \hat E_n $ originates from \cite{KKP21}. For any $ \bx\in[0,1]^d $ and $ n\in\N $, let $ l_n(\bx)\in\R_{\ge 0} $ be the non-negative number such that $ \mu(R(\bx,l_n(\bx)\br_n))=r_{n,1}\cdots r_{n,d} $. Let
\[\xi_n(\bx):=l_n(\bx)\br_n\in(\R_{\ge 0})^d.\]
Then $ R(\bx, \xi_n(\bx)) $ is a hyperrectangle obtained by scaling $ R(\bx, \br_n) $ by a factor $ l_n(\bx) $. Define
\[\hat E_n:=\{\bx\in[0,1]^d:T^n\bx\in R(\bx, \xi_n(\bx))\}\quad\text{and}\quad \hat \rc(\{\br_n\}):=\limsup\hat E_n.\]

The next lemma decribes some local structures of $ \hat E_n $ which are similar to those of $ E_n $.

\begin{lem}\label{l:simple lemma s}
	Let $ R(\bx,\br) $ be a hyperrectangle with center $ \bx\in [0,1]^d $ and $ \br=(r_1,\dots,r_d)\in(\R_{\ge 0})^d $. Then, for any $ n\in\N $ with $ \br_n\ne \mathbf{0} $ and any subset $ F $ of $ R(\bx,\br) $, we have
	\[F\cap T^{-n}R(\bx,\xi_n(\bx)-2t\br_n)\subset F\cap\hat E_n\subset F\cap T^{-n} R(\bx,\xi_n(\bx)+2t\br_n),\]
	where $ t=\max_{1\le i\le d}(r_i/r_{n,i}) $.
\end{lem}
\begin{proof}
	Since $ r_i=(r_i/r_{n,i})\cdot r_{n,i}\le tr_{n,i} $, for any $ \bz\in F $, we have
	\begin{equation}\label{eq:z in larger rectangle}
		\bz\in F\subset R(\bx,\br)\subset R(\bx,t\br_n).
	\end{equation}
	Then, by the triangle inequality
	\[R(\bz,\xi_n(\bx)-t\br_n)\subset R(\bx, \xi_n(\bx))\subset R(\bz,\xi_n(\bx)+t\br_n).\]
	Hence,
	\[\mu(R(\bz,\xi_n(\bx)-t\br_n))\le\mu(R(\bx, \xi_n(\bx)))=r_{n,1}\cdots r_{n,d}\le \mu(R(\bz,\xi_n(\bx)+t\br_n)).\]
	By the definitions of $ l_n(\bz) $ and $ \xi_n(\bz) $, the above inequalities imply that
	\[l_n(\bx)-t\le l_n(\bz)\le l_n(\bx)+t,\]
and so
\begin{equation}\label{eq:radius estimations of y}
	R(\bz,\xi_n(\bx)-t\br_n)\subset R(\bz,\xi_n(\bz))\subset R(\bz,\xi_n(\bx)+t\br_n).
\end{equation}
	If $ \bz $ also belongs to $ \hat E_n $, then $ T^n\bz\in R(\bz, \xi_n(\bz)) $. We then deduce from \eqref{eq:z in larger rectangle} and \eqref{eq:radius estimations of y} that
	\[T^n\bz\in R(\bz, \xi_n(\bz))\subset R(\bz,\xi_n(\bx)+t\br_n)\subset R(\bx,\xi_n(\bx)+2t\br_n).\]
	That is, $ \bz\in T^{-n}R(\bx,\xi_n(\bx)+2t\br_n) $. Therefore,
	\[F\cap \hat E_n\subset F\cap T^{-n}R(\bx,\xi_n(\bx)+2t\br_n).\]

	The first inclusion follows similarly.
\end{proof}

We need to calculate the measures of the hyperrectangles appearing in Lemma \ref{l:simple lemma s}.
\begin{lem}\label{l:measure of enlarged ball}
	Write $ s=1-1/q $. For any $ \br=(r_1,\dots,r_d)\in(\R_{\ge 0})^d $, we have
	\begin{align*}
		\mu(R(\bx,\xi_n(\bx)+\br))&\le r_{n,1}\cdots r_{n,d}+c_1\cdot \max_{1\le i\le d}r_i^s,\\
		\mu(R(\bx,\xi_n(\bx)-\br))&\ge r_{n,1}\cdots r_{n,d}-c_1\cdot\max_{1\le i\le d}r_i^s,
	\end{align*}
	where $ c_1=2d\|h\|_q $, and $ \|h\|_q:=(\int |h|^q\dif\lm^d)^{1/q} $ is the $ L^q $-norm of $ h $.
\end{lem}
\begin{proof}
	We prove the first inequality only, as the second one follows similarly. For any measurable set $ F\subset [0,1]^d $, by H\"older's inequality, we have
		\begin{equation}\label{eq:Lq norm cs}
			\mu(F)=\int\chi_F\dif\mu=\int\chi_F\cdot h\dif\lm^d\le \|h\|_q\cdot\lm^d(F)^{1-1/q}=\|h\|_q\cdot\lm^d(F)^{s}.
		\end{equation}
		Since the annulus $ R(\bx,\xi_n(\bx)+\br)\setminus R(\bx,\xi_n(\bx)) $ is contained in $ 2d $ hyperrectangles, each of them having volume less than $ \max_{1\le i\le d}r_i $, by \eqref{eq:Lq norm cs} we have
	\[\begin{split}
		\mu(R(\bx,\xi_n(\bx)+\br))&= \mu(R(\bx,\xi_n(\bx)))+\mu(R(\bx,\xi_n(\bx)+\br)\setminus R(\bx,\xi_n(\bx)))\\
		&\le r_{n,1}\cdots r_{n,d}+\|h\|_q\cdot \lm^d(R(\bx,\xi_n(\bx)+\br)\setminus R(\bx,\xi_n(\bx)))^s\\
		&\le r_{n,1}\cdots r_{n,d}+\|h\|_q \cdot 2d\cdot\max_{1\le i\le d}r_i^s.\qedhere
	\end{split}\]
\end{proof}

We state and prove some estimates on the measure of $ \hat E_n $ that are necessary for the proof of the divergence part of Theorem \ref{t:main}.
\begin{lem}\label{l:measure of hat En}
	Let $ B\subset [0,1]^d $ be a ball. Then, for any large $ n $,
	\[\frac{1}{2}\mu(B)\cdot r_{n,1}\cdots r_{n,d}\le\mu(B\cap \hat E_n)\le 2\mu(B)\cdot r_{n,1}\cdots r_{n,d}.\]
\end{lem}
\begin{proof}

	If $ \br_n=\mathbf{0} $, then $ \mu(B\cap \hat E_n)=0 $ and the lemma follows.

	Now, suppose $ \br_n\ne\mathbf{0} $. Recall that a ball here is with respect to the maximum norm and thus corresponds to a Euclidean hypercube.
	Denote the radius of $ B $ by $ r_0 $. Partition $ B $ into $ r_0^d\cdot e^{\tau n/2} $ balls with radius $ r:=e^{-\tau n/(2d)} $. The collection of these balls is denoted by
	\[\{B(\bx_i,r):1\le i\le r_0^d e^{\tau n/2}\}.\]
	 Let $ t=\max_{1\le i\le d}(r/r_{n,i}) $. By Lemma \ref{l:simple lemma s}, we have
	\begin{equation}\label{eq:inclusion1}
		B\cap \hat E_n=\bigcup_{i\le r_0^de^{\tau n/2}}B(\bx_i,r)\cap \hat E_n\supset \bigcup_{i\le r_0^de^{\tau n/2}}B(\bx_i,r)\cap T^{-n} R(\bx_i,\xi_n(\bx_i)-2t\br_n).
	\end{equation}
	Applying the exponential mixing property and Lemma \ref{l:measure of enlarged ball}, we have
	\begin{align}
		\mu(B\cap\hat E_n)&\ge \sum_{i\le r_0^de^{\tau n/2}}\mu(B(\bx_i,r)\cap T^{-n} R(\bx_i,\xi_n(\bx_i)-2t\br_n))\notag\\
		&\ge \sum_{i\le r_0^de^{\tau n/2}}\big(\mu (B(\bx_i,r))-ce^{-\tau n}\big)\mu(R(\bx_i,\xi_n(\bx_i)-2t\br_n))\notag\\
		&\ge \sum_{i\le r_0^de^{\tau n/2}}\big(\mu (B(\bx_i,r))-ce^{-\tau n}\big)(r_{n,1}\cdots r_{n,d}-c_1(2t)^s)\notag\\
		&=(\mu(B)-cr_0^de^{-\tau n/2})(r_{n,1}\cdots r_{n,d}-c_1(2t)^s),\label{eq:hat En lower bound}
	\end{align}
where we use $ \lim_{n\to\infty} |\br_n|=0<1 $ (see equation \eqref{eq:rn to 0}) in the third inequality.
	Since $ \br_n\ne\mathbf{0} $, by our assumption on the sequence $ (\br_n) $ we have $ r_{n,i}\ge n^{-2} $ for $ 1\le i\le d $. Hence,
	\[t=\max_{1\le i\le d}\frac{r}{r_{n,i}}\le n^2e^{-2\tau n/d}.\]
	 Substituting the upper bound for $ t $ to \eqref{eq:hat En lower bound}, we deduce that for all large $ n $,
	\[\mu(B\cap \hat E_n)\ge \frac{1}{2}\mu(B)\cdot r_{n,1}\cdots r_{n,d}.\]

	The upper estimation for $ \mu(\hat E_n) $ can be proved by replacing \eqref{eq:inclusion1} with
	\[
	B\cap \hat E_n=\bigcup_{i\le r_0^de^{\tau n/2}}B(\bx_i,r)\cap \hat E_n\subset \bigcup_{i\le r_0^de^{\tau n/2}}B(\bx_i,r)\cap T^{-n} R(\bx_i,\xi_n(\bx_i)+2t\br_n).\qedhere\]
\end{proof}

\subsubsection{Estimating the measure of $ B\cap \hat E_m\cap \hat E_n $ with $ m<n $}\label{s:correlation}
\

We proceed to estimate the correlations of the sets $ \hat E_n $. Let $ m,n\in \N $ with $ m<n $. If $ \br_m=\textbf{0} $ or $ \br_n=\textbf{0} $, then $ \mu(\hat E_m)=\mu(\hat E_n)=0 $. Thus, for any ball $ B $, we have
\[\mu(B\cap \hat E_m\cap \hat E_n)=0.\]

Here and below, we assume that neither $ \br_m $ nor $ \br_n $ is $ \textbf{0} $. Then, according to the assumptions at the beginning of Section \ref{s:divergence}, we have $ r_{m,i}> m^{-2} $ and $ r_{n,i}>n^{-2} $ for $ 1\le i\le d $.

The next lemma is needed to show that the sets in question satisfy the bounded property $ \pro $.
\begin{lem}[{\cite[Lemma 3.2.38]{Federer} and \cite[Proposition 3.3.5]{Krantz}}]\label{l:boundary of C1 image}
	Suppose that $ f\colon\R^{d}\to\R^d $ is $ C^1 $ and the boundary of $ A $ is included in a $ C^1 $ piecewise embedded compact submanifold of codimension one. Then,
	\[\mc(\partial (f(A)))\le \Big(\sup_{\bx\in A}\|D_{\bx}f\|\Big)^{d-1}\mc(\partial A).\]
\end{lem}
Now, we use the estimate form Lemma \ref{l:boundary of C1 image} to deduce the following lemma. Recall the notation $ \mathcal F_n $ in Remark \ref{r:1}.
\begin{lem}\label{l:boundary estimation}
	Let $ T\colon[0,1]^d\to [0,1]^d $ be a piecewise expanding map. For any $ n\ge 1 $, any $ J_n\in \mathcal F_n $, and any hyperrectangles $ R_1, R_2\subset [0,1]^d $, we have
	\[\mc\big(\partial (J_n\cap R_1\cap T^{-n}R_2)\big)\le 4d+ \frac{K}{1-L^{-(d-1)}},\]
	where $ L $ and $ K $ are constants given in Definition \ref{d:C1 map}.
\end{lem}
\begin{proof}
	Suppose $ d=1 $. In this case both $ R_1 $ and $ R_2 $ are intervals. Since each $ U_i $ is connected, this means that $ U_i $ is an open interval. Definition \ref{d:C1 map} (i) indicates that $ T $ is strictly monotonic and continuous on each $ U_i $. It is easily verified that $ J_n\cap R_1\cap T^{-n}R_2 $ is either empty or an interval. Hence,
	\[\mc(\partial(J_n\cap R_1\cap T^{-n}R_2))\le 2.\]

	Now, suppose $ d\ge 2 $. We claim that for any $ J_n\in\mathcal F_n $,
	\begin{equation}\label{eq:boundary of cylinder}
		\mc(\partial J_n)\le K\frac{1-L^{-(d-1)(n-1)}}{1-L^{-(d-1)}}.
	\end{equation}
	We proceed by induction. For $ n=1 $, this is implied by Definition \ref{d:C1 map} (ii).

	Assume that \eqref{eq:boundary of cylinder} holds for some $ n\ge 1 $. We will prove that \eqref{eq:boundary of cylinder} holds for $ n+1 $. Write $ J_{n+1}=J_{n}\cap T^{-n}(U_{i_n}) $, where $ J_n\in\mathcal F_n $ and $ 1\le i_n\le P $. Since $ T $ can be extended to a $ C^1 $ map on each $ \overline{U_i} $, the map $ T^n|_{J_{n}}\colon J_{n}\to T^nJ_{n} $ can also be extended to a $ C^1 $ map on $ \overline{J_n} $. Moreover, the inverse of $ T^n|_{J_{n}} $, denoted by $ T^{-n}|_{T^nJ_n} $, exists and is also a $ C^1 $ map. By \eqref{eq:expanding constant}, we have
	\begin{equation}\label{eq:inverse bounded derivative}
		\sup_{\bx\in J_n}\|D_\bx(T^{-n}|_{T^nJ_n})\|\le L^{-n}.
	\end{equation}
	Since $ T^{-n}|_{T^nJ_n} $ is invertible and $ C^1 $, we have
	\begin{align}
		\partial J_{n+1}&=\partial (J_{n}\cap T^{-n}(U_{i_n}))\notag\\
		&=\partial (T^{-n}|_{T^nJ_{n}}(T^nJ_{n}\cap U_{i_n}))\notag\\
		&=T^{-n}|_{T^nJ_{n}}(\partial(T^nJ_{n}\cap U_{i_n})).\label{eq:boundary inverse}
	\end{align}
	Clearly, $ \partial(T^nJ_{n}\cap U_{i_n})\subset \partial (T^nJ_n)\cup \partial U_{i_n} $. The boundary $ \partial(T^nJ_{n}\cap U_{i_n}) $ can be decomposed into a union of two closed sets $ A_1=\partial(T^nJ_{n}\cap U_{i_n})\cap\partial (T^nJ_n) $ and $ A_2=\partial(T^nJ_{n}\cap U_{i_n})\cap\partial U_{i_n} $. Substituting this decomposition into \eqref{eq:boundary inverse}, we obtain
	\[\begin{split}
		\partial J_{n+1}=T^{-n}|_{T^nJ_{n}}(A_1\cup A_2)&= T^{-n}|_{T^nJ_{n}}(A_1)\cup T^{-n}|_{T^nJ_{n}}(A_2)\\
		&\subset T^{-n}|_{T^nJ_{n}}(\partial(T^nJ_n))\cup T^{-n}|_{T^nJ_{n}}(A_2)\\
		&=\partial J_n\cup T^{-n}|_{T^nJ_{n}}(A_2).
	\end{split}\]
	Since both $ \partial J_n $ and $ A_2 $ are $ C^1 $ piecewise embedded compact submanifolds of codimension one, it follows from Lemma \ref{l:boundary of C1 image} and \eqref{eq:inverse bounded derivative} that
	\begin{align}
		\mc(\partial J_{n+1})&\le\mc(\partial J_n\cup T^{-n}|_{T^nJ_{n}}(A_2))\notag\\
		&\le \mc(\partial J_n)+L^{-(d-1)n}\mc(A_2)\notag\\
		&\le \mc(\partial J_n)+L^{-(d-1)n}\mc(\partial U_{i_n})\label{eq:boundary upper bound}\\
		&\le \mc(\partial J_n)+KL^{-(d-1)n}\notag.
	\end{align}
	Using the inductive hypothesis, we  prove the claim.

	By the same reason as \eqref{eq:boundary upper bound}, we have
	\[\begin{split}
		\mc(\partial (J_n\cap R_1\cap T^{-n}R_2))&\le \mc(\partial R_1)+\mc(\partial (J_n\cap T^{-n}R_2))\\
		&\le 2d+\mc(\partial J_n)+L^{-(d-1)n}\mc(\partial R_2)\\
		&\le 2d+K\frac{1-L^{-(d-1)(n-1)}}{1-L^{-(d-1)}}+2dL^{-(d-1)n}\\
		&\le 4d+\frac{K}{1-L^{-(d-1)}},
	\end{split}\]
	which completes the proof of the lemma.
\end{proof}

\begin{rem}\label{r: relax condition on T}
	When $ d=1 $, the condition on $ T $ stated in Definition \ref{d:C1 map} (i) can be relaxed to that $ T $ is strictly monotonic and continuous on each $ U_i $. We can draw the same conclusion as Lemma \ref{l:boundary estimation} under this weaker condition.
\end{rem}
Since the sets under consideration satisfy the bounded property $ \pro $, we have the following first correlation estimate.

\begin{lem}\label{l:estimations on correlations 1}
	Let $ B\subset [0,1]^d $ be a ball. Then, there exists a constant $ c_2 $ such that for any sufficiently large integers $ m $ and $ n $ with $ 8d^2\log n/(s\tau)\le m<n $, we have
	\begin{equation*}
		\mu(B\cap \hat E_m\cap\hat E_n)\le c_2\mu(B)(r_{m,1}\cdots r_{m,d}+e^{-\tau (n-m)})\cdot r_{n,1}\cdots r_{n,d}.
	\end{equation*}
\end{lem}
\begin{proof}
	Let $ r_0 $ be the radius of $ B $. Partition $ B $ into $ r_0^de^{\tau m/2} $ balls with radius $ r:=e^{-\tau m/(2d)} $. The collection of these balls is denoted by
	\[\{B(\bx_i,r):1\le i\le r_0^de^{\tau m/2}\}.\]

	Let $ t_m=\max_{1\le i\le d}(r/r_{m,i}) $ and $ t_n=\max_{1\le i\le d}(r/r_{n,i}) $. For each ball $ B(\bx_i,r) $, by Lemma \ref{l:simple lemma s} we have
	\[B(\bx_i,r)\cap \hat E_m\cap \hat E_n\subset B(\bx_i,r)\cap T^{-m}R(\bx_i,\xi_m(\bx_i)+2t_m\br_m)\cap T^{-n}R(\bx_i,\xi_n(\bx_i)+2t_n\br_n).\]
	Hence,
	\begin{equation}\label{eq:B(x,r) cap Em cap En}
		\begin{split}
			&B\cap\hat E_m\cap\hat E_n\\
			&\subset \bigcup_{i\le r_0^de^{\tau m/2}}B(\bx_i,r)\cap T^{-m}R(\bx_i,\xi_m(\bx_i)+2t_m\br_m)\cap T^{-n}R(\bx_i,\xi_n(\bx_i)+2t_n\br_n).
		\end{split}
	\end{equation}
	For notational simplicity, write
	\[R_i(m)=R(\bx_i,\xi_m(\bx_i)+2t_m\br_m)\quad\text{and}\quad R_i(n)=R(\bx_i,\xi_n(\bx_i)+2t_n\br_n).\]
	By Lemma \ref{l:boundary estimation}, the sets $ J_{n-m}\cap R_i(m)\cap T^{-(n-m)}R_i(n) $ satisfy the bounded property $ \pro $.
	Applying the exponential mixing property, we have
	\begin{align}
		&\mu\big(B(\bx_i,r)\cap T^{-m}R_i(m)\cap T^{-n}R_i(n)\big)\label{eq:exponential mixing twice}\\
		=&\sum_{J_{n-m}\in\mathcal F_{n-m}} \mu\big(B(\bx_i,r)\cap T^{-m}(J_{n-m}\cap R_i(m)\cap T^{-(n-m)}R_i(n))\big)\notag\\
		\le &\sum_{J_{n-m}\in\mathcal F_{n-m}}\big(\mu(B(\bx_i,r))+ce^{-\tau m}\big) \mu(J_{n-m}\cap R_i(m)\cap T^{-(n-m)}R_i(n))\notag\\
		=&\big(\mu(B(\bx_i,r))+ce^{-\tau m}\big) \mu(R_i(m)\cap T^{-(n-m)}R_i(n))\notag\\
		\le& \big(\mu(B(\bx_i,r))+ce^{-\tau m}\big)\big(\mu(R_i(m))+ce^{-\tau (n-m)}\big) \mu(R_i(n))\label{eq:three balls cap}.
	\end{align}
	By Lemma \ref{l:measure of enlarged ball}, \eqref{eq:three balls cap} is majorized by
	\[ \big(\mu(B(\bx_i,r))+ce^{-\tau m}\big)(r_{m,1}\cdots r_{m,d}+c_1(2t_m)^s+ce^{-\tau (n-m)})(r_{n,1}\cdots r_{n,d}+c_1(2t_n)^s).\]
	Summing over $ i\le r_0^de^{\tau m/2} $, we have
	\begin{equation}\label{eq:1}
		\begin{split}
			\mu(B\cap\hat E_m\cap \hat E_n)\le &\big(\mu(B)+cr_0^de^{-\tau m/2}\big)(r_{m,1}\cdots r_{m,d}+c_1(2t_m)^s+ce^{-\tau (n-m)})\\
			&\cdot(r_{n,1}\cdots r_{n,d}+c_1(2t_n)^s).
		\end{split}
	\end{equation}
	Since $ r_{m,i}> m^{-2} $ and $ r_{n,i}> n^{-2} $, it follows from $ r=e^{-\tau m/(2d)} $ that
	\[t_m=\max_{1\le i\le d} \frac{e^{-\tau m/(2d)}}{r_{m,i}}\le m^{2}e^{-\tau m/(2d)},\]
	and from $ 8d^2\log n/(s\tau)\le m<n  $ and $ s=1-1/q<1 $ that
	\[t_n^s=\bigg(\max_{1\le i\le d} \frac{e^{-\tau m/(2d)}}{r_{n,i}}\bigg)^s\le n^{2s}e^{-s\tau m/(2d)} \le n^{2s}e^{-4d\log n}\le n^{2-4d}\le n^{-2d}\le r_{n,1}\cdots r_{n,d}.\]
	Thus, there exists a constant $ c_2 $ such that for any sufficiently large integers $ m $ and $ n $ with $ 8d^2\log n/(s\tau)\le m<n $, we have
	\begin{equation*}
		\mu(B\cap \hat E_m\cap\hat E_n)\le c_2\mu(B)(r_{m,1}\cdots r_{m,d}+e^{-\tau (n-m)})\cdot r_{n,1}\cdots r_{n,d}.\qedhere
	\end{equation*}
\end{proof}

Kirsebom, Kunde and Persson \cite[Lemma 4.2]{KKP21} obtained some similar estimates of the correlations of  the sets $ \hat E_n $ when $ d=1 $. As hinted in their paper, these estimates are not sufficient to conclude the divergence part of Theorem \ref{t:main}. To get around this barrier, we need the following lemma, which helps us  estimate the correlations of the sets $ \hat E_n $ for the case $ m\le 8d^2\log n/(s\tau) $.

\begin{lem}\label{l:measure of intersection}
	Let $ T\colon[0,1]^d\to [0,1]^d $ be a piecewise expanding map. Then, there exists a constant $ c_3 $ such that for any integers $ m $ and $ n $ with $ m\le 8d^2\log n/(s\tau) $, and any hyperrectangles $ R_1,R_2, R_3\subset [0,1]^d $, we have
	\[\mu(R_1\cap T^{-m}R_2\cap T^{-n}R_3)\le \big(\mu(R_1\cap T^{-m}R_2)+c_3e^{-\tau n/2}\big)\mu(R_3).\]
\end{lem}
\begin{proof}
	It is clear that
	\[\mu(R_1\cap T^{-m}R_2\cap T^{-n}R_3)=\sum_{J_m\in\mathcal F_m}\mu(J_m\cap R_1\cap T^{-m}R_2\cap T^{-n}R_3).\]
	By Lemma \ref{l:boundary estimation}, the sets $ J_m\cap R_1\cap T^{-m}R_2 $ satisfy the bounded property $ \pro $. Applying the exponential mixing property, we have
	\begin{align}
		&\sum_{J_m\in\mathcal F_m}\mu(J_m\cap R_1\cap T^{-m}R_2\cap T^{-n}R_3)\notag\\
		\le &\sum_{J_m\in\mathcal F_m}\big(\mu(J_m\cap R_1\cap T^{-m}R_2)+ce^{-\tau n}\big)\mu(R_3)\notag\\
		=&\big(\mu(R_1\cap T^{-m}R_2)+ \# \mathcal F_m\cdot ce^{-\tau n}\big)\mu(R_3).\label{eq:exponentiall decay intersection of three balls}
	\end{align}
	Since $ m\le 8d^2\log n/(s\tau) $, we have
	\[ \# \mathcal F_m\le P^m\le n^{8d^2\log P/(s\tau)}.\]
	Applying this upper bound for $  \#  \mathcal F_m $ to \eqref{eq:exponentiall decay intersection of three balls}, we complete the proof.
\end{proof}

We are ready to estimate the correlations of the sets $ \hat E_n $ for the case $ m\le 8d^2\log n/(s\tau) $.

\begin{lem}\label{l:estimations on correlations 2}
	Let $ T\colon[0,1]^d\to [0,1]^d $ be a piecewise expanding map. Let $ B\subset [0,1]^d $ be a ball. Then, there exists a constant $ c_4 $ such that for all integers $ m $ and $ n $ with $ m\le 8d^2\log n/(s\tau) $, we have
	\[\mu(B\cap\hat E_m\cap \hat E_n)\le c_4(\mu(B)+e^{-\tau m/2})\cdot r_{m,1}\cdots r_{m,d}\cdot r_{n,1}\cdots r_{n,d}.\]
	In particular, if $ m\ge -\frac{2\log \mu(B)}{\tau} $, then
	\[\mu(B\cap\hat E_m\cap \hat E_n)\le 2c_4\mu(B)\cdot r_{m,1}\cdots r_{m,d}\cdot r_{n,1}\cdots r_{n,d}.\]
\end{lem}
\begin{proof}
	Denote by $ r_0 $ the radius of $ B $. Partition $ B $ into $ r_0^d e^{\tau n/4} $ balls with radius $ r_1:=e^{-\tau n/(4d)} $. The collection of these balls is denoted by
	\[
	\{B(\bx_i,r_1):1\le i\le r_0^de^{\tau n/4}\}.\]

	Let $ u_m=\max_{1\le i\le d}(r_1/r_{m,i}) $ and $ u_n=\max_{1\le i\le d}(r_1/r_{n,i}) $. For each ball $ B(\bx_i,r_1) $, by Lemma \ref{l:simple lemma s} we have
	\[B(\bx_i,r_1)\cap\hat E_m\cap\hat E_n\subset B(\bx_i,r_1)\cap T^{-m}R(\bx_i,\xi_m(\bx_i)+2u_m\br_m)\cap T^{-n}R(\bx_i,\xi_n(\bx_i)+2u_n\br_n).\]
	 Hence,
	\begin{equation}\label{eq:B(x,r) cap Em cap En s}
		B\cap\hat E_m\cap\hat E_n\subset \bigcup_{i\le r_0^de^{\tau n/4}}G_i\cap T^{-n}R(\bx_i,\xi_n(\bx_i)+2u_n\br_n),
	\end{equation}
	where
	\[G_i=B(\bx_i,r_1)\cap T^{-m}R(\bx_i,\xi_m(\bx_i)+2u_m\br_m).\]
 By \eqref{eq:B(x,r) cap Em cap En s} and Lemmas \ref{l:measure of enlarged ball} and \ref{l:measure of intersection} we have
\begin{align}
	\mu(B\cap\hat E_m\cap\hat E_n)&\le \sum_{i\le r_0^de^{\tau n/4}}\mu(G_i\cap T^{-n}R(\bx_i,\xi_n(\bx_i)+2u_n\br_n))\notag\\
	&\le\sum_{i\le r_0^de^{\tau n/4}} \big(\mu(G_i)+c_3e^{-\tau n/2}\big)\mu(R(\bx_i,\xi_n(\bx_i)+2u_n\br_n))\notag\\
	&\le\sum_{i\le r_0^de^{\tau n/4}} \big(\mu(G_i)+c_3e^{-\tau n/2}\big)(r_{n,1}\dots r_{n,d}+c_1(2u_n)^s)\notag\\
	&=\Big(\sum_{i\le r_0^de^{\tau n/4}}\mu(G_i)+c_3r_0^de^{-\tau n/4}\Big)(r_{n,1}\dots r_{n,d}+c_1(2u_n)^s)\label{eq:summation need to estimate}.
\end{align}

Now, we are in a position to bound the summation in \eqref{eq:summation need to estimate} from above. Partition $ [0,1]^d $ into $ r_0^de^{\tau m/2} $ balls with radius $ r_2:=e^{-\tau m/(2d)} $. Denote the collection of these balls by
\[\{B(\bz_j,r_2):1\le j\le r_0^de^{\tau m/2}\}.\]

Let $ v_m=\max_{1\le i\le d}(r_2/r_{m,i}) $. By definition, one has $ v_m\ge u_m $. For any $ \bx\in B(\bz_j,r_2)\cap G_i $, we have
\[T^m\bx\in R(\bx_i,\xi_m(\bx_i)+2u_m\br_m)\quad\text{and}\quad \bx_i\in B(\bz_j,r_1+r_2)\subset R(\bz_j,(u_m+v_m)\br_m).\]
Hence,
\begin{align}
	T^m\bx\in R(\bx_i,\xi_m(\bx_i)+2u_m\br_m)&\subset R(\bz_j,\xi_m(\bx_i)+3u_m\br_m+v_m\br_m)\notag\\
	&\subset R(\bz_j,\xi_m(\bx_i)+4v_m\br_m)\label{eq:Tmx-zj}.
\end{align}

On the other hand, since
\[ R(\bz_j,\xi_m(\bx_i)-2v_m\br_m)\subset R(\bz_j,\xi_m(\bx_i)-(u_m+v_m)\br_m)\subset R(\bx_i,\xi_m(\bx_i)),\]
we have
\[\mu(R(\bz_j,\xi_m(\bx_i)-2v_m\br_m))\le\mu(R(\bx_i,\xi_m(\bx_i)))=r_{m,1}\cdots r_{m,d},\]
which means that
\[l_m(\bz_j)\ge l_m(\bx_i)-2v_m.\]
 This together with \eqref{eq:Tmx-zj} gives that
\[T^m\bx_i\in  R(\bz_j,\xi_m(\bx_i)+4v_m\br_m)\subset R(\bz_j,\xi_m(\bz_j)+6v_m\br_m),\]
and further that
\[B(\bz_j,r_2)\cap B(\bx_i,r_1)\cap T^{-m}R(\bx_i,\xi_m(\bx_i)+2u_m\br_m)\subset B(\bz_j,r_2)\cap T^{-m}R(\bz_j,\xi_m(\bz_j)+6v_m\br_m).\]
Therefore,
\begin{equation}\label{eq:subset well estimate}
	\begin{split}
	&\bigcup_{i\le r_0^de^{\tau n/4}} B(\bx_i,r_1)\cap T^{-m}R(\bx_i,\xi_m(\bz_j)+2u_m\br_m)\\
		&\hspace{2 em}\subset\bigcup_{j\le r_0^de^{\tau m/2}} B(\bz_j,r_2)\cap T^{-m}R(\bz_j,\xi_m(\bz_j)+6v_m\br_m).
	\end{split}
\end{equation}
With this inclusion and noting that the top of \eqref{eq:subset well estimate} is a disjoint union, we have
\begin{align}
	&\sum_{i\le r_0^de^{\tau n/4}}\mu(B(\bx_i,r_1)\cap T^{-m}R(\bx_i,\xi_m(\bx_i)+2u_m\br_m))\notag\\
	&=\mu\Big(\bigcup_{i\le r_0^de^{\tau n/4}} B(\bx_i,r_1)\cap T^{-m}R(\bx_i,\xi_m(\bx_i)+2u_m\br_m)\Big)\notag\\
	&\le \mu\Big(\bigcup_{j\le r_0^de^{\tau m/2}}B(\bz_j,r_2)\cap T^{-m}R(\bz_j,\xi_m(\bx_i)+6v_m\br_m)\Big)\notag\\
	&\le \sum_{j\le r_0^de^{\tau m/2}} \mu(B(\bz_j,r_2)\cap T^{-m}R(\bz_j,\xi_m(\bx_i)+6v_m\br_m)).\label{eq:second to the last step}
\end{align}
By the exponential mixing property and Lemma \ref{l:measure of enlarged ball},
\begin{align}
	&\sum_{j\le r_0^de^{\tau m/2}} \mu(B(\bz_j,r_2)\cap T^{-m}R(\bz_j,\xi_m(\bx_i)+6v_m\br_m))\notag\\
	&\le \sum_{j\le r_0^de^{\tau m/2}}(\mu(B(\bz_j,r_2))+ce^{-\tau m}) \mu(R(\bz_j,\xi_m(\bx_i)+6v_m\br_m))\notag\\
	&\le \sum_{j\le r_0^de^{\tau m/2}}(\mu(B(\bz_j,r_2))+ce^{-\tau m})(r_{m,1}\cdots r_{m,d}+c_1(6v_m)^s)\notag\\
	&= (\mu(B)+cr_0^de^{-\tau m/2})(r_{m,1}\cdots r_{m,d}+c_1(6v_m)^s).\label{eq:final estimation}
\end{align}
Combining with \eqref{eq:summation need to estimate}, \eqref{eq:second to the last step} and \eqref{eq:final estimation}, we have
\[\begin{split}
	\mu(B\cap \hat E_m\cap\hat E_n)\le &\big( (\mu(B)+cr_0^de^{-\tau m/2})(r_{m,1}\cdots r_{m,d}+c_1(6v_m)^s)+c_3r_0^de^{-\tau n/4}\big)\\
	&\cdot(r_{n,1}\dots r_{n,d}+c_1(2u_n)^s).
\end{split}\]
Note that $ r_{m,i}>m^{-2} $ and $ r_{n,i}>n^{-2} $. It follows from $ r_1=e^{-\tau n/(4d)} $ and $ r_2=e^{-\tau m/(2d)} $ that
\[v_m=\max_{1\le i\le d}\frac{e^{-\tau m/(2d)}}{r_{m,i}}\le m^{2}e^{-\tau m/(2d)}\qquad\text{and}\qquad u_n=\max_{1\le i\le d}\frac{e^{-\tau n/(4d)}}{r_{n,i}}\le n^{2}e^{-\tau n/(4d)}.\]
Hence, there exists a constant $ c_4 $ such that
\[\mu(B\cap \hat E_m\cap \hat E_n)\le c_4(\mu(B)+e^{-\tau m/2})\cdot r_{m,1}\cdots r_{m,d}\cdot r_{n,1}\cdots r_{n,d}.\qedhere\]
\end{proof}
\subsubsection{Completing the proof of Theorem \ref{t:main}}\label{s:full measure}
\

The main ingredient in proving $ \mu(\hat\rc(\{\br_n\}))=1 $ is the use of the following measure-theoretical result \cite[Lemma 7]{BeDiVe06}, which is a generalization of a known result for Lebesgue measure. Although the original result is only valid for doubling measures, it actually holds for all Borel probability measures on $ \R^d $.
\begin{lem}[{\cite[Lemma 7]{BeDiVe06}}]\label{l:density lemma}
	Let $ \nu $ be a Borel probability measure on $ \R^n $. Let $ E $ be a measurable subset of $ \R^d $. Assume that there are constants $ r_0 $ and $ \alpha>0 $ such that for any ball $ B:=B(\bx, r) $ with $ r<r_0 $, we have $ \nu(B\cap E)\ge \alpha\nu(B) $. Then,
	\[\nu(E)=1.\]
\end{lem}
\begin{proof}
	Assume by contradiction that $ \nu(E)<1 $. Then, the complement of $ E $ has positive measure. By the density theorem \cite[Corollary 2.14]{Mattila}, we have
	\[\lim_{r\to 0}\frac{\nu(B(\bx,r)\cap E^c)}{\nu(B(\bx,r))}=1\quad\text{for }\nu\text{-a.e.\,}\bx\in E^c.\]
	For any such $ \bx $, there exists $ r<r_0 $ small enough such that
	\[\nu(B(\bx,r)\cap E^c)\ge (1-\alpha)\nu(B(\bx,r))\]
	or equivalently
	\[\nu(B(\bx,r)\cap E)<\alpha\mu(B(\bx,r)),\]
	which contradicts the assumption.
\end{proof}
The above lemma together with the Paley-Zygmund inequality gives $ \mu(\hat \rc(\{\br_n\}))=1 $.

\begin{lem}\label{l:hat R full measure}
	There exists a constant $ \alpha_1 $ such that for any ball $ B\subset [0,1]^d $, we have
	\[\mu(B\cap \hat\rc(\{\br_n\}))\ge \alpha_1\mu(B).\]
	In particular, $ \mu(\hat\rc(\{\br_n\}))=1 $.
\end{lem}
\begin{proof}
	Fix a ball $ B\subset [0,1]^d $ with $ \mu(B)>0 $. We can replace at most finitely many $ \br_n $'s by $ \bf 0 $, so that the estimations given in Lemmas \ref{l:measure of hat En}, \ref{l:estimations on correlations 1} and \ref{l:estimations on correlations 2} hold for all $ n\in \N $. This manipulation does not change the property $ \sum_{n\ge 1} r_{n,1}\cdots r_{n,d}=\infty $ and the resulting set is smaller than the original one.

	Let $ N\in\N $ and $ Z_N(x)=\sum_{n=1}^{N}\chi_{B\cap \hat E_n}(x) $. By Lemma \ref{l:measure of hat En}, one has
	\begin{equation}\label{eq:estimation sum N mu hat En}
		\frac{1}{2}\mu(B)\sum_{n=1}^N r_{n,1}\cdots r_{n,d}\le \mathbb E(Z_N)=\sum_{n=1}^N\mu(B\cap \hat E_n)\le 2\mu(B)\sum_{n=1}^N r_{n,1}\cdots r_{n,d}.
	\end{equation}

	On the other hand,
	\begin{align}
		\mathbb E (Z_N^2)&=\sum_{m,n=1}^N\mu(B\cap \hat E_m\cap \hat E_n)\notag\\
		&=2\sum_{n=1}^N\sum_{m=1}^{n-1}\mu(B\cap \hat E_m\cap \hat E_n)+\sum_{n=1}^N\mu(B\cap \hat E_n).\label{eq:second moment}
	\end{align}
	By Lemmas \ref{l:estimations on correlations 1} and \ref{l:estimations on correlations 2},
	\begin{align}
		\sum_{n=1}^N\sum_{m=1}^{n-1}\mu(B\cap \hat E_m\cap \hat E_n)&=\sum_{n=1}^N\sum _{m= \frac{8d^2}{s\tau}\log n}^{n-1}\mu(B\cap \hat E_m\cap \hat E_n)+\sum_{n=1}^N\sum_{m=1}^{\frac{8d^2}{s\tau}\log n-1}\mu(B\cap\hat E_m\cap \hat E_n)\notag\\
		&\le \sum_{n=1}^N\sum _{m= \frac{8d^2}{s\tau}\log n}^{n-1}c_2\mu(B)\cdot r_{n,1}\cdots r_{n,d}\cdot (r_{m,1}\cdots r_{m,d}+e^{-\tau (n-m)})\notag\\
		&\quad+\sum_{n=1}^N\sum_{m=1}^{\frac{8d^2}{s\tau}\log n-1}2c_4\mu(B)\cdot r_{m,1}\cdots r_{m,d}\cdot r_{n,1}\cdots r_{n,d}\notag.
	\end{align}
	Hence, there exist $ c_5 $ and $ c_6 $ such that
	\begin{equation}\label{eq:estimate second moment}
		\begin{split}
			&\sum_{n=1}^N\sum_{m=1}^{n-1}\mu(B\cap \hat E_m\cap \hat E_n)\\
			\le&  c_5\mu(B)\sum_{n=1}^N\sum_{m=1}^{n-1}r_{m,1}\cdots r_{m,d}\cdot r_{n,1}\cdots r_{n,d}+c_6\mu(B)\sum_{n=1}^Nr_{n,1}\cdots r_{n,d}.
		\end{split}
	\end{equation}
	Then by the Paley-Zygmund inequality, for any $ \lambda>0 $, we deduce from \eqref{eq:estimation sum N mu hat En}, \eqref{eq:second moment} and \eqref{eq:estimate second moment} that
	\[\begin{split}
		\mu(Z_N>\lambda \mathbb E(Z_N))&\ge (1-\lambda)^2 \frac{\mathbb E(Z_N)^2}{\mathbb E(Z_N^2)}\\
		&\ge (1-\lambda)\frac{\big(\frac{\mu(B)}{2}\sum_{n=1}^Nr_{n,1}\cdots r_{n,d}\big)^2}{2c_5\mu(B)\big(\sum_{n=1}^N r_{n,1}\cdots r_{n,d}\big)^2+(2+2c_6)\sum_{n=1}^Nr_{n,1}\cdots r_{n,d}}.
	\end{split}\]
	By letting $ N\to\infty $, we get
	\begin{equation}\label{eq:positive measure}
		\begin{split}
			\mu(B\cap \hat\rc(\{\br_n\}))&=\mu(\limsup (B\cap\hat E_n))\ge  \mu(\limsup(Z_N>\lambda\mathbb E(Z_N)))\\
			&\ge\limsup \mu(Z_N>\lambda\mathbb E(Z_N))>(1-\lambda)\frac{\mu(B)}{8c_5}.
		\end{split}
	\end{equation}
Let $ \lambda=1/2 $ and $ \alpha_1=(16c_5)^{-1} $. Then the first point of the lemma holds. The second point follows naturally from Lemma \ref{l:density lemma}.
\end{proof}

As a consequence, we have the following result.
\begin{lem}
	The set $ \rc(\{\br_n\}) $ is of full $ \mu $-measure.
\end{lem}
\begin{proof}
	Take a sequence of positive numbers $ \{a_n\} $ such that
	\[\sum_{n=1}^{\infty}\frac{r_{n,1}\cdots r_{n,d}}{a_n^d}=\infty\quad\text{and}\quad\lim_{n\to\infty}a_n=\infty.\]
	Let $ \tilde \br_n=a_n^{-1}\br_n $. Applying Lemma \ref{l:hat R full measure} to $ \hat\rc(\{\tilde \br_n\}) $, we have that for $ \mu $-almost every $ \bx $,
	\begin{equation}\label{eq:final equation}
		T^n\bx\in R(\bx,\tilde\xi_n(\bx))\quad\text{for infinitely many }n\in\N,
	\end{equation}
	where $ \tilde \xi_n(\bx):=\tilde l_n(\bx)\tilde \br_n $ and $ \tilde l_n(\bx)\in\R_{\ge 0} $ is the positive number such that $ \mu(R(\bx,\tilde l_n(\bx)\tilde \br_n))=r_{n,1}\cdots r_{n,d}/a_n^d $.

	By Zygmund differentiation theorem (see Theorem \ref{t:differentiation theorem}), for $ \mu $-almost every $ \bx $
	\[\lim_{n\to\infty}\frac{\mu(R(\bx,\br_n))}{2^d\cdot r_{n,1}\cdots r_{n,d}}=h(\bx).\]
	For any such $ \bx $, since $ \lim_{n\to\infty} a_n=\infty $, there exists $ N(\bx) $ such that for all $ n\ge N(\bx) $,
	\[\mu(R(\bx,\tilde\xi_n(\bx)))=r_{n,1}\cdots r_{n,d}/a_n^d<2\cdot 2^dh(\bx)\cdot r_{n,1}\cdots r_{n,d}.\]
	Increasing the integer $ N(\bx) $ if necessary, we have
	\[\mu(R(\bx,\tilde\xi_n(\bx)))<\mu(R(\bx,\br_n))\quad \text{for all }n\ge N(\bx).\]
	Therefore,
	\[R(\bx,\tilde\xi_n(\bx))\subset R(\bx,\br_n).\]
	This together with \eqref{eq:final equation} yields that for $ \mu $-almost every $ \bx $,
	\[T^n\bx\in R(\bx,\br_n)\quad\text{for infinitely many }n\in\N,\]
	which completes the proof.
\end{proof}

\section{Proof of Theorem \ref{t:main s}}\label{s:hyperbola}
Recall that for any $ \bx\in[0,1]^d $ and $ \delta>0 $, $ H(\bx, \delta) $ is given by
\[H(\bx, \delta)=\bx+\{\bz\in [-1,1]^d: z_1\cdots z_d<\delta \}.\]
Let $ (\delta_n) $ be as in Theorem \ref{t:main s}. Then, $ \rc^\times(\{\delta_n\})=\limsup D_n $, where
\[D_n:=\{\bx\in [0,1]^d:T^n\bx\in H(\bx, \delta_n)\}.\]

We summarize some standard facts about the geometric properties of $ H(\bx,\delta) $, which will be used later.
Clearly, for any $ \bx\in[0,1]^d $, there exists a constant $ K_1 $ such that
\begin{equation}\label{eq:boundary H}
	\sup_{0\le \delta\le1}\mc(\partial H(\bx, \delta))\le K_1.
\end{equation}
Moreover, following the idea of \cite[Lemma 4]{LLVZ22}, we can deduce that for any $ 0<\delta<r<1 $, we have
\begin{equation}\label{eq:3}
	\lm^d(B(\bx,r)\cap H(\bx,\delta))=2^d\delta\bigg(\sum_{t=0}^{d-1}\frac{1}{t!}\Big(\log \frac{r^d}{\delta}\Big)^t\bigg).
\end{equation}
For future use, we give some consequence of \eqref{eq:3}. For any $ \delta>0 $, it holds that
\begin{equation}\label{eq:measure H(delta) upper bound}
	\lm^d(H(\bx,\delta))\le d2^d\delta(-\log\delta)^{d-1},
\end{equation}
and that
\begin{equation}\label{eq:measure H(delta) lower bound}
	\lm^d(B(\bx,r)\cap H(\bx,\delta))\ge \delta(-\log\delta)^{d-1}\quad \text{for }r^d>\sqrt \delta.
\end{equation}
In particular, if $ \delta=\delta_n+c_0n^{-3} $ for some $ c_0<2^{d+2} $, then for any large $ n $ with $ \delta_n>n^{-2} $,
\begin{equation}\label{eq:measure H(delta) upper bound s}
	\lm^d(H(\bx,\delta_n+c_0n^{-3}))\le d4^{d}\delta_n(-\log \delta_n)^{d-1}.
\end{equation}

Throughout this section, we fix the collection $ \mathcal C_2 $ of subsets $ F\subset [0,1]^d $ satisfying the bounded property
\[\pros: \qquad\sup_{F\in\mathcal C_2}\mc(\partial F)<2d+K_1+\frac{K}{1-L^{-(d-1)}},\]
where $ K_1 $ is given in \eqref{eq:boundary H}, and $ L $ and $ K $ are given in Definition \ref{d:C1 map}. It is easily seen that hyperrectangles and the hyperboloids $ H(\bx,\delta) $ satisfy the bounded property $ \pros $.

Similar to the proof of Theorem \ref{t:main}, we will frequently use the following lemma.
\begin{lem}\label{l:simple lemma 2}
	Let $ B(\bx,r)\subset[0,1]^d $ be a ball with center $ \bx\in [0,1]^d $ and radius $ r>0 $. Then for any $ F\subset B(\bx,r) $,
	\[F\cap T^{-n}H_n(\bx,\delta_n-2^dr)\subset F\cap D_n\subset F\cap T^{-n}H(\bx,\delta_n+2^dr).\]
\end{lem}
\begin{proof}
	For any $ \bz\in F $, there exists $ \br=(r_1,\dots,r_d)\in \R^d $ with $ |\br|<r $ such that $ \bz=\bx+\br $. If $ \bz $ also belongs to $ D_n $, then $ T^n\bz\in H(\bz,\delta_n) $. Write $ T^n\bz=(z_1',\dots,z_d') $. Then,
	\[|z_1'-z_1|\cdots|z_d'-z_d|<\delta_n,\]
	which implies
	\begin{align}
		|z_1'-x_1|\cdots|z_d'-x_d|&=|z_1'-z_1+r_1|\cdots|z_d'-z_d+r_d|\notag\\
		&\le |z_1'-z_1|\cdots|z_d'-z_d|+2^d r\notag\\
		&<\delta_n+2^dr\label{eq:4}.
	\end{align}
	That is, $ T^n\bz\in H(\bx, \delta_n+2^dr) $. Therefore,
	$ \bz\in F\cap T^{-n}H(\bx, \delta_n+2^dr) $.

	The first inclusion follows similarly.
\end{proof}

\subsection{Convergence part}
Assume that $ \sum_{n\ge 1}\delta_n(-\log \delta_n)^{d-1}<\infty $. Under the setting of Theorem \ref{t:main s}, we will prove $ \mu(\rc^\times(\{\delta_n\}))=0 $.

Since the density $ h $ of $\mu$ is bounded, the convergence part can be established without using Zygmund differentiation theorem.

\begin{proof}[Proof of Theorem \ref{t:main s}: convergence part]
	Let $ n\ge 1 $. Partition $ [0,1]^d $ into $ (n/2)^{2d} $ balls with radius $ n^{-2} $. Denote  by
	\[
	\{B(\bx_i,n^{-2}):1\le i\le (n/2)^{2d}\}\] the collection of these balls.
	By Lemma \ref{l:simple lemma 2}, we have
	\[D_n=\bigcup_{i\le (n/2)^{2d}}B(\bx_i,n^{-2})\cap D_n\subset\bigcup _{i\le (n/2)^{2d}}B(\bx_i,n^{-2})\cap T^{-n}H(\bx_i,\delta_n+2^dn^{-2}).\]
	Thus, by the exponential mixing property of $ \mu $,
	\[\begin{split}
		\mu(D_n)&\le \sum_{i\le (n/2)^{2d}}\mu(B(\bx_i,n^{-2})\cap T^{-n}H(\bx_i,\delta_n+2^dn^{-2}))\\
		&\le \sum_{i\le (n/2)^{2d}}\big(\mu(B(\bx_i,n^{-2}))+ce^{-\tau n}\big)\mu(H(\bx_i,\delta_n+2^dn^{-2})).
	\end{split}\]
Since $ h\le \mathfrak{c} $, by \eqref{eq:measure H(delta) upper bound} the above sum is majorized by
\[\begin{split}
	&\sum_{i\le (n/2)^{2d}}\big(\mu(B(\bx_i,n^{-2}))+ce^{-\tau n}\big) \mathfrak{c}\cdot d2^d(\delta_n+2^dn^{-2})(-\log(\delta_n+2^dn^{-2}))^{d-1}\\
	=&\mathfrak{c}d2^d(1+c(n/2)^{2d}e^{-\tau n})(\delta_n+2^dn^{-2})(-\log(\delta_n+2^dn^{-2}))^{d-1}\\
	\le & \mathfrak{c}d2^d(1+c(n/2)^{2d}e^{-\tau n})(\delta_n(-\log \delta_n)^{d-1}+2^dn^{-2}(-\log n^{-2})^{d-1}).
\end{split}\]
Therefore,
\[\begin{split}
	\sum_{n=1}^{\infty}\mu(D_n)&\le \sum_{n=1}^{\infty}\mathfrak{c}d2^d(1+c(n/2)^{2d}e^{-\tau n})(\delta_n(-\log \delta_n)^{d-1}+2^dn^{-2}(-\log n^{-2})^{d-1})\\
	&<\infty.
\end{split}\]
By Borel--Cantelli lemma,
\[\mu(\rc^\times(\{\delta_n\}))=0.\qedhere\]
\end{proof}
\subsection{Divergence part}
In the sequel, assume that $ \sum_{n\ge 1}\delta_n(-\log \delta_n)^{d-1}=\infty $. Following the same discussion in Section 3.1, we assume that $ \delta_n $ is either greater than $ n^{-2} $ or equal to $ 0 $. The proof of the divergence part uses a similar but more direct approach to the one used in proving Theorem \ref{t:main}.

Let $ V $ be the open set with $ \mu(V)=1 $ given in Theorem \ref{t:main s}. Then $ V $ can be written as
\[V=\bigcup_{\delta>0}V_\delta,\]
where $ V_{\delta}:=\{\bx\in V: B(\bx,\delta)\in V\} $. It is easily seen that $ V_\delta $ is monotonic with respect to $ \delta $.

This subsection aims at proving the following lemma.
\begin{lem}\label{l:modulo}
	For any $ \delta>0 $ there exists a constant $ \alpha_2=\alpha_2(\delta)>0 $ such that for all $ r<\delta/2 $ and $ \bx\in V_\delta $
	\begin{equation}
		\mu(B(\bx,r)\cap \rc^\times(\{\delta_n\}))\ge\alpha_2\mu(B(\bx,r)).
	\end{equation}
\end{lem}

Now, we deduce the divergence part of Theorem \ref{t:main s} from Lemma \ref{l:modulo}.

\begin{proof}[Proof of the divergence part of Theorem \ref{t:main s} modulo Lemma \ref{l:modulo}]
	For any $ \delta>0 $ with $ \mu(V_\delta) >0 $, we claim that
	\[\mu(V_\delta \cap\rc^\times(\{\delta_n\}))=\mu(V_\delta ).\]
	Assume by contradiction that $ \mu(V_\delta \cap\rc^\times(\{\delta_n\}))<\mu(V_\delta ) $. Then, the set $ A:=V_\delta\setminus\rc^\times(\{\delta_n\})  $ has positive $\mu$-measure. By the density theorem \cite[Corollary 2.14]{Mattila}, for $ \mu $-almost every $ \bx\in A $
	\[\lim_{r\to 0}\frac{\mu(B(\bx,r)\cap A)}{\mu(B(\bx,r))}=1.\]
	For any such $ \bx $, there exists $ r<\delta/2 $ small enough such that
	\[\mu(B(\bx,r)\cap A)\ge (1-\alpha_2)\mu(B(\bx,r)).\]
	This means that
	\[\mu(B(\bx,r)\cap \rc^\times(\{\delta_n\}))<\alpha_2 \mu(B(\bx,r)),\]
	which contradicts Lemma \ref{l:modulo}.

	Since $ V=\cup_{\delta>0}V_\delta $ and $ \mu(V)=1 $, we arrive at the conclusion.
\end{proof}

\subsubsection{Estimating the measure of $ \mu(B\cap D_n) $}
\begin{lem}\label{l:measure of hat Dn}
	Under the setting of Theorem \ref{t:main s}. Let $ B $ be a ball with center $ \bz\in V_\delta  $ and radius $ r<\delta/2 $. Then, for any large $ n $,
	\[\frac{\mu(B)}{\mathfrak{c}2^{d+1}}\cdot \delta_n(-\log \delta_n)^{d-1}\le\mu(B\cap D_n)\le \mathfrak{c} d4^{d+1}\mu(B)\cdot \delta_n(-\log \delta_n)^{d-1}.\]
\end{lem}
\begin{proof}
	If $ \delta_n=0 $, then $ \mu(B\cap D_n)=0 $ and the lemma follows.

	Suppose $ \delta_n\ne 0 $. Then, $ \delta_n> n^{-2} $.
	Since the radius of $ B $ is less than $ \delta/2 $, by definition one has $ B(\bx,\delta/2)\subset V $ for any $ \bx\in B $.
	Note that the density $ h $, when restricted on $ V $, is bounded from below by $ \mathfrak{c}^{-1} $. It follows from \eqref{eq:measure H(delta) lower bound} that for any $ \bx\in B $ and $ n\in\N $ with $ n<(\delta/2)^{-d} $,
	\[\begin{split}
		\mu(H(\bx,\delta_n-2^dn^{-3}))&\ge
		\mu(B(\bx,\delta/2)\cap H(\bx,\delta_n-2^dn^{-3}))\\
		&\ge\mathfrak{c}^{-1}\lm^d(B(\bx,\delta/2)\cap H(\bx,\delta_n-2^dn^{-3}))\\
		&\ge  \mathfrak{c}^{-1} (\delta_n-2^dn^{-3})(-\log(\delta_n-2^dn^{-3}))^{d-1}.
	\end{split}\]

	Partition $ B $ into $ r^d n^{3d} $ balls with radius $ n^{-3} $. The collection of these balls is denoted by $ \{B(\bx_i,n^{-3}):1\le i\le r^dn^{3d}\} $. By Lemma \ref{l:simple lemma 2}, we have
	\begin{equation}
		B\cap D_n\supset\bigcup_{i\le r^dn^{3d}}B(\bx_i,n^{-3})\cap T^{-n}H(\bx_i,\delta_n-2^dn^{-3}).
	\end{equation}
	By the exponential mixing property, it follows that
	\begin{align}
		\mu(B\cap D_n)&\ge \sum_{i\le r^dn^{3d}}\mu(B(\bx_i,n^{-3})\cap T^{-n}H(\bx_i,\delta_n-2^dn^{-3}))\notag\\
		&\ge \sum_{i\le r^dn^{3d}}\big(\mu (B(\bx_i,n^{-3}))-ce^{-\tau n}\big)\mu(H(\bx_i,\delta_n-2^dn^{-3}))\notag\\
		&\ge \sum_{i\le r^dn^{3d}}\big(\mu (B(\bx_i,n^{-3}))-ce^{-\tau n}\big)\cdot\mathfrak{c}^{-1} (\delta_n-2^dn^{-3})(-\log(\delta_n-2^dn^{-3}))^{d-1}\notag\\
		&=(\mu(B)-cr^dn^{3d}e^{-\tau n/2})\cdot \mathfrak{c}^{-1} (\delta_n-2^dn^{-3})(-\log(\delta_n-2^dn^{-3}))^{d-1}\notag.
	\end{align}

On the other hand, for large $ n $, we have
\[cr^dn^{3d}e^{-\tau n/2}\le \mu(B)/2.\]
Futher, since $ \delta_n\ge n^{-2} $, for large $ n $, we have
\[\delta_n-2^dn^{-3}\ge \delta_n/2\quad\text{and}\quad -\log(\delta_n-2^dn^{-3})\ge (-\log \delta_n)/2.\]
Hence, for large $ n $,
\[\mu(B\cap D_n)\ge\frac{\mu(B)}{\mathfrak{c}2^{d+1}}\cdot \delta_n(-\log \delta_n)^{d-1}.\]

The second inequality follows in the same way if we use the inclusion
\[B\cap D_n\subset \bigcup_{i\le r^dn^{3d}}B(\bx_i,n^{-3})
\cap T^{-n}H(\bx_i,\delta_n+2^dn^{-3}).\qedhere\]
\end{proof}
\begin{rem}
	The set $ V $ being open is crucial for us because it gives a uniform lower estimation for the $\mu$-measure of a hyperboloid centered at $ V_\delta  $.
\end{rem}
\subsubsection{Estimating the measure of $ B\cap D_m\cap D_n $ with $ m<n $}
\

We proceed to estimate the correlations of the sets $ D_n $.  Let $ m,n\in \N $ with $ m<n $. If $ \delta_m=0 $ or $ \delta_n=0 $, then $ \mu(D_m)=\mu(D_n)=0 $. Thus, for any ball $ B\subset [0,1]^d $, we have
\[\mu(B\cap D_m\cap D_n)=0.\]

Suppose for the moment that neither $ \delta_m $ nor $ \delta_n $ is $ 0 $.
\begin{lem}\label{l:estimations on correlations s1}
	Let $ B\subset [0,1]^d $ be a ball. There exists a constant $ \tilde c_1 $ such that for all large integers $ m,n $ with $ (4d\log n)/\tau\le m<n $,
	\begin{equation*}
		\mu(B\cap D_m\cap D_n)\le \tilde c_1\mu(B) (\delta_m(-\log \delta_m)^{d-1}+e^{-\tau (n-m)})\cdot \delta_n(-\log \delta_n)^{d-1}.
	\end{equation*}
\end{lem}
\begin{proof}
	Let $ r $ be the radius of $ B $. Partition $ B $ into $ r^dn^{3d} $ balls with radius $ n^{-3} $. Let $ \{B(\bx_i,n^{-3}):1\le i\le r^dn^{3d}\} $ denote the collection of these balls. For each ball $ B(\bx_i,n^{-3}) $, by Lemma \ref{l:simple lemma 2} we have
	\[B(\bx_i,n^{-3})\cap D_m\cap D_n\subset B(\bx_i,n^{-3})\cap T^{-m} H(\bx_i,\delta_m+2^dn^{-3})\cap T^{-n} H(\bx_i,\delta_n+2^dn^{-3}).\]
	Hence,
	\begin{equation}\label{eq:B(x,r) cap Ｄm cap Ｄn}
		B{\cap}D_m{\cap}D_n\subset \bigcup_{i\le r^dn^{3d}}B(\bx_i,n^{-3})\cap T^{-m} H(\bx_i,\delta_m+2^dn^{-3})\cap T^{-n} H(\bx_i,\delta_n+2^dn^{-3}).
	\end{equation}
	By the same reason as \eqref{eq:three balls cap}, it follows from \eqref{eq:measure H(delta) upper bound s} that
	\begin{align}
		&\mu\big(B(\bx_i,n^{-3})\cap T^{-m} H(\bx_i,\delta_m+2^dn^{-3})\cap T^{-n} H(\bx_i,\delta_n+2^dn^{-3})\big)\label{eq:2}\\
		\le& \big(\mu(B(\bx_i,n^{-3})+ce^{-\tau m}\big)\big(\mu( H(\bx_i,\delta_m+2^dn^{-3}))+ce^{-\tau (n-m)}\big)\notag\\
		&\hspace{20em}\cdot \mu( H(\bx_i,\delta_n+2^dn^{-3}))\notag\\
		\le& \big(\mu(B(\bx_i,n^{-3})+ce^{-\tau m}\big)\big(\mathfrak{c}d2^d\delta_m(-\log \delta_m)^{d-1}+ce^{-\tau (n-m)}\big)\notag\\
		&\hspace{20em}\cdot\mathfrak{c}d2^d\delta_n(-\log \delta_n)^{d-1}\notag.
	\end{align}
	Summing over $ i\le r^dn^{3d} $, we have
	\begin{align}
		&\mu(B\cap D_m\cap D_n)\notag\\
		\le& \big(\mu(B)+cr^dn^{3d}e^{-\tau m}\big)\big(\mathfrak{c}d2^d\delta_m(-\log \delta_m)^{d-1}+ce^{-\tau (n-m)}\big)\cdot \mathfrak{c}d2^d\delta_n(-\log \delta_n)^{d-1}.\notag
	\end{align}
	Since $ m\ge (4d\log n)/\tau $, we have
	\[n^{3d}e^{-\tau m}\le n^{3d}\cdot n^{-4d}=n^{-d}.\]
	Hence, there exists a constant $ \tilde c_1 $ such that for all large integers $ m $ and $ n $ with $ (4d\log n)/\tau\le m<n $,
	\begin{equation*}
		\mu(B\cap D_m\cap D_n)\le \tilde c_1\mu(B)(\delta_m(-\log \delta_m)^{d-1}+e^{-\tau (n-m)})\cdot \delta_n(-\log \delta_n)^{d-1},
	\end{equation*}
	which finishes the proof.
\end{proof}

In light of Lemma \ref{l:boundary estimation} and \eqref{eq:boundary H}, one can verify that the sets $ J_n\cap B\cap T^{-n}H(\bx,\delta)$ satisfy the bound property $ \pros $, where $ J_n\in \mathcal{F}_n $ and $ B $ is a ball. As a result, we have the following estimations, which may be proved in much the same way as Lemma \ref{l:measure of intersection}.

\begin{lem}\label{l:measure of intersection s}
	Let $ T\colon[0,1]^d\to [0,1]^d $ be a piecewise expanding map. There exists a constant $ \tilde c_2 $ such that for any integers $ m\le(4d\log n)/\tau $, and for any ball $ B $, $ \bx\in[0,1]^d $ and $ 0<\delta_1,\delta_2<1 $, we have
	\[\mu(B\cap T^{-m}H(\bx,\delta_1)\cap T^{-n}H(\bx,\delta_2))\le \big(\mu(B\cap T^{-m}H(\bx,\delta_1))+\tilde c_2e^{-\tau n/2}\big)\mu(H(\bx,\delta_2)).\]
\end{lem}

With Lemma \ref{l:measure of intersection s} at our disposal, we are ready to estimate the correlations of the sets $ D_n $ for the case $ m\le (4d\log n)/\tau $.

\begin{lem}\label{l:estimations on correlations s2}
	Let $ B\subset [0,1]^d $ be a ball. There exists a constant $ \tilde c_3 $ such that for all integers $ m,n $ with $ m< (4d\log n)/\tau $, we have
	\[\mu(B\cap D_m\cap D_n)\le \tilde c_3\mu(B)\cdot \delta_m(-\log \delta_m)^{d-1}\cdot \delta_n(-\log \delta_n)^{d-1}.\]
\end{lem}
\begin{proof}
	We follow the notation of Lemma \ref{l:estimations on correlations s1}. By Lemma \ref{l:measure of intersection s} and \eqref{eq:measure H(delta) upper bound}, the measure in \eqref{eq:2} can be estimated as
	\[\begin{split}
		&\mu\big(B(\bx_i,n^{-3})\cap T^{-m} H(\bx_i,\delta_m+2^dn^{-3})\cap T^{-n} H(\bx_i,\delta_n+2^dn^{-3})\big)\\
		\le& \big(\mu(B(\bx_i,n^{-3})\cap T^{-m} H(\bx_i,\delta_m+2^dn^{-3}))+\tilde c_2e^{-\tau n/2}\big)\mu (H(\bx_i,\delta_n+2^dn^{-3}))\\
		\le&\big(\mu(B(\bx_i,n^{-3})\cap T^{-m} H(\bx_i,\delta_m+2^dn^{-3}))+\tilde c_2e^{-\tau n/2}\big)\cdot\mathfrak{c}d2^d\delta_n(-\log \delta_n)^{d-1}.
	\end{split}\]
	Summing over $ i $, we have
	\begin{align}
		&\mu(B\cap D_m\cap D_n)\notag\\
		\le& \Big(\sum_{i\le r^dn^{3d}}\mu(B(\bx_i,n^{-3})\cap T^{-m} H(\bx_i,\delta_m+2^dn^{-3}))+\tilde c_2r^dn^{3d}e^{-\tau n}/2\Big)\label{eq:summation needed to estimate}\\
		&\cdot\mathfrak{c}d2^d\delta_n(-\log \delta_n)^{d-1}.\notag
	\end{align}
	Now, the task is to estimate the summation in \eqref{eq:summation needed to estimate}. Partition $ B $ into $ r^dm^{3d} $ balls with radius $ m^{-3} $. Denote the collection of these balls by
	\[\{B(\bz_j,m^{-3}):1\le j\le r^dm^{3d}\}.\]
	For any $ \bx\in B(\bz_j,m^{-3})\cap B(\bx_i,n^{-3})\cap T^{-m} H(\bx_i,\delta_m+2^dn^{-3}) $, we have
	\[T^m\bx\in H(\bx_i,\delta_m+2^dn^{-3})\quad\text{and}\quad \bx_i\in B(\bz_j,m^{-3}+n^{-3}).\]
	In the spirit of Lemma \ref{l:simple lemma 2}, we obtain
	\[T^m\bx\in H(\bx_i, \delta_m+2^dn^{-3})\subset H(\bz_j, \delta_m+2^d m^{-3}+2\cdot 2^dn^{-3})\subset H(\bz_j, 2^{d+2}m^{-3}).\]
	That is, $ \bx\in B(\bz_j,m^{-3})\cap T^{-m}H(\bz_j,\delta_m+2^{d+2}m^{-3}) $. Hence,
	\begin{equation}\label{eq:subset well estimate 1}
		\begin{split}
			&\bigcup_{i\le r^dn^{3d}} B(\bx_i,n^{-3})\cap T^{-m} H(\bx_i,\delta_m+2^dn^{-3})\\
			&\hspace{2 em}\subset\bigcup_{j\le r^dm^{3d}} B(\bz_j,m^{-3})\cap T^{-m}H(\bz_j,\delta_m+2^{d+2}m^{-3}).
		\end{split}
	\end{equation}
	With this inclusion and noting that the top of \eqref{eq:subset well estimate 1} is a disjoint union, we have
	\begin{align}
		&\sum_{i\le r^dn^{3d}}\mu(B(\bx_i,n^{-3})\cap T^{-m} H(\bx_i,\delta_m+2^dn^{-3}))\notag\\
		=&\mu\Big(\bigcup_{i\le r^dn^{3d}} B(\bx_i,n^{-3})\cap T^{-m} H(\bx_i,\delta_m+2^dn^{-3})\Big)\notag\\
		\le&\sum_{j\le r^dm^{3d}} \mu(B(\bz_j,m^{-3})\cap T^{-m}H(\bz_j,\delta_m+2^{d+2}m^{-3})).\notag
	\end{align}
	By the exponential mixing property, the above summation is majorized by
	\begin{align}
		&\sum_{j\le r^dm^{3d}} \mu(B(\bz_j,m^{-3})\cap T^{-m}H(\bz_j,\delta_m+2^{d+2}m^{-3}))\notag\\
		\le& \sum_{j\le r^dm^{3d}}(\mu(B(\bz_j,m^{-3}))+ce^{-\tau m}) \mu (H(\bz_j,\delta_m+2^{d+2}m^{-3}))\notag\\
		\le &\sum_{j\le r^dm^{3d}}(\mu(B(\bz_j,m^{-3}))+ce^{-\tau m})\cdot\mathfrak{c}d2^d\delta_m(-\log \delta_m)^{d-1}\notag\\
		=& (\mu(B)+cr^dm^{3d}e^{-\tau m})\cdot\mathfrak{c}d2^d\delta_m(-\log \delta_m)^{d-1}\notag\\
		\le& \mu(B)(1+\mathfrak{c}cm^{3d}e^{-\tau m})\cdot\mathfrak{c}d2^d\delta_m(-\log \delta_m)^{d-1}\label{eq:second to the last step s},
	\end{align}
	where the last inequality follows from $ \mu(B)\ge \mathfrak{c}^{-1}r^d $.
	Combining with \eqref{eq:summation needed to estimate} and \eqref{eq:second to the last step s}, we have
	\[\begin{split}
		\mu(B\cap D_m\cap D_n)\le &\big(\mu(B)(1+\mathfrak{c}cm^{3d}e^{-\tau m})\cdot\mathfrak{c}d2^d\delta_m(-\log \delta_m)^{d-1}+\tilde c_2r^dn^{3d}e^{-\tau n}/2\big)\\
		&\cdot\mathfrak{c}d2^d\delta_n(-\log \delta_n)^{d-1}.
	\end{split}\]
	Since $ \delta_m>m^{-2} $ and $ \delta_n>n^{-2} $, there exists a constant $ \tilde c_3 $ such that
	\[\mu(B\cap D_m\cap D_n)\le \tilde c_3\mu(B)\cdot \delta_m(-\log \delta_m)^{d-1}\cdot \delta_n(-\log \delta_n)^{d-1}.\qedhere\]
\end{proof}
\begin{proof}[Proof of Lemma \ref{l:modulo}]
	Applying the same technique used in proof of Lemma \ref{l:hat R full measure}, we arrive at the conclusion.
\end{proof}

\section{Hausdorff dimension of $ \rc(\Psi) $}
The proof of the Hausdorff dimension of $ \rc(\Psi) $ makes essential use of the ideas from the proofs of \cite[Theorems 7 and 12]{LLVZ22} for shrinking target problems. Thus, in the subsequent proofs, we will make an effort to maintain consistent notation and employ similar arguments to improve clarity and readability.


\subsection{Upper bound for $ \hdim\rc(\Psi) $}
The upper bound is relatively easier to prove by using the definition of Hausdorff dimension on the natural cover of $ \rc(\Psi) $.

We begin with some notation and known results. For $ |\beta|>1 $, let $ T_\beta $ be the $\beta$-transformation on $ [0,1) $ defined by
\[T_\beta x=\beta x\pmod 1.\]
%
Let
\[\cq:=\bigg\{\Big[0,\frac{1}{|\beta|}\Big), \dots,\Big[\frac{k}{|\beta|},\frac{k+1}{|\beta|}\Big),\dots, \Big[\frac{\lfloor|\beta|\rfloor}{|\beta|},1\Big) \bigg\}\]
be the partition of $ [0,1) $ corresponding to the $ T_\beta $. For each $ n\ge 1 $, define
\[\cq^n:=\big\{Q_{i_0}\cap T_\beta^{-1}Q_{i_1}\cap\cdots\cap T^{-(n-1)}Q_{i_{n-1}}:Q_{i_j}\in\cq\text{ for $ 0\le j\le n-1 $}\big\}.\]
Clearly, $\cq^n$ consists of pairwise disjoint intervals, each with a length less than $|\beta|^{-n}$. These intervals are called cylinders of order $n$. Moreover, for each $ I\in \cq^n $, $ T_\beta^n|_I $ is linear with slope $ \beta^n $.

Denote by $ N_n $ the cardinality of $ \cq^n $. Since the topological entropy of $ T_\beta $ is $ \log|\beta| $, for any $ \epsilon>0 $, we have
\begin{equation}\label{eq:Nn upper bound}
	N_n\le |\beta|^{n(1+\epsilon)}
\end{equation}
for $ n $ sufficiently large.

Now we turn to bound $ \hdim\rc(\Psi) $ from above. Let $ \cq^n=\{C_{n,j}:1\le j\le N_n\} $. Given a cylinder $ C_{n,j} $ of order $ n $. Let $ r>0 $ and consider the set
\[F_{n,j}:=\{x\in C_{n,j}: |T_\beta^nx-x|<r\}.\]
Choose two points $ x,y\in C_{n,j} $ such that $ x\in F_{n,j} $ and $ |x-y|>3r/|\beta|^n $. By using the triangle inequality, for any $ n $ with $ 3|\beta|^{-n}<1 $, we have
\[\begin{split}
	|T_\beta^ny-y|&=|T_\beta^ny-T_\beta^nx+T_\beta^nx-x+x-y|\\
	&\ge |T_\beta^ny-T_\beta^nx|-|T_\beta^nx-x|-|x-y|\\
	&=|\beta|^n|x-y|-|T_\beta^nx-x|-|x-y|\\&> 3r-r-3r/|\beta|^n\\
	&>r.
\end{split}\]
Hence, $ y\notin F_{n,j} $. This implies that $ F_{n,j} $ is contained in an interval, denoted by $ I_{n,j} $, whose length is $ 6r/|\beta|^n $. Thus,
\begin{equation*}
	\{x\in[0,1):|T_\beta^nx-x|<r\}\subset \bigcup_{j=1}^{N_n}I_{n,j},
\end{equation*}
where each $ I_{n,j} $ is an interval lying in some cylinder of order $ n $. Therefore, for $ 1\le i\le d $, we have
\begin{equation}\label{eq:<psiisub}
	\{x\in[0,1):|T_{\beta_i}^nx-x|<\psi_i(n)\}\subset \bigcup_{j_i=1}^{N_{i,n}}I_{n,j_i}^{(i)},
\end{equation}
where each $ I_{n,j_i}^{(i)} $ is an interval of length $ 6\psi_i(n)|\beta_i|^{-n} $ contained in some cylinder of order $ n $ and $ N_{i,n} $ is the number of such intervals. For $ n\in \N $, let
\[J_n:=\{\textbf{j}=(j_1,\dots,j_d):1\le j_i\le N_{i,n}~(1\le i\le d)\}\]
and for $ \textbf{j}\in J_n $, let
\[R_{n,\textbf{j}}:=I_{n,j_1}^{(1)}\times\cdots\times I_{n,j_d}^{(d)}.\]
Let
\[R_n=\bigcup_{\textbf{j}\in J_n}R_{n,\textbf{j}}\qquad\text{and}\qquad \rc_n=\{R_{n,\textbf{j}}:\textbf{j}\in J_n\}.\]
	Then, by the definition of $ \rc(\Psi) $ and \eqref{eq:<psiisub},
	\[\rc(\Psi)\subset\bigcap_{N=1}^\infty\bigcup_{n=N}^\infty R_n\]
	and for any $ N\ge 1 $,
	\[\rc(\Psi)\subset \bigcup_{n=N}^\infty R_n.\]
	Namely, $ \rc(\Psi) $ can be covered by the family $ \{\rc_n:n=N,N+1,\dots\} $ of rectangles.

	Now, we will estimate the number of balls $ B_{i,n} $ of diameter $ 6\psi_i(n)|\beta_i|^{-n} $ (the sidelength of the rectangle in $ \rc_n  $ along the direction of the $ i $-th axis) needed to cover the set $ R_n $. This estimation, along with the definition of Hausdorff dimension, provides an upper bound for $ \hdim\rc(\Psi) $. To this end, we adopt the method similar to that presented in the proof of \cite[Proposition 4]{LLVZ22}, but with slight improvements. Such improvements enable us to bypass Bugeaud and Wang's result \cite[Theorem 1.2]{BuWa2014} and to obtain the upper bound without assuming $ \beta_i>1 $ for all $ 1\le i\le d $.

	Given a rectangle $ R=R_{n,\textbf{j}}\in\rc_n $ with $ \textbf{j}=(j_1,\dots, j_d) $. Since $ |I_{n, j_k}^{(k)}|=6\psi_k(n)|\beta_k|^{-n} $, we see that $ I_{n, j_k}^{(k)} $ can be covered by
	\[\max\bigg\{1,\frac{\psi_k(n)|\beta_k|^{-n}}{\psi_i(n)|\beta_i|^{-n}}\bigg\}\]
	intervals of length $ 6\psi_i(n)|\beta_i|^{-n} $. Thus, we can find a collection $ \cb_{i,n}(R) $ of balls $ B_{i,n} $ that covers $ R $ with
	\begin{equation}\label{eq:Bni(R) upper bound}
			 \#  \cb_{i,n}(R)\le \prod_{\substack{1\le k\le d:\\ \psi_k(n)|\beta_k|^{-n}\ge \psi_i(n)|\beta_i|^{-n}}}\frac{\psi_k(n)|\beta_k|^{-n}}{\psi_i(n)|\beta_i|^{-n}}.
		\end{equation}
Note that if $ |\beta_k|^{-n}< 6\psi_i(n)|\beta_i|^{-n} $ for some $ 1\le k\le d $, then the ball $ B_{i,n} $ may intersect multiple rectangles in $ \rc_n $ along the direction of the $ k $-th axis. In this case, since
	\[\bigcup_{j_k'=1}^{N_{k,n}}R_{n,(j_1,\dots,j_{k-1},j_k',j_{k+1},\dots, j_d)}\subset I_{n,j_1}^{(1)}\times\cdots\times I_{n,j_{k-1}}^{(k-1)}\times [0,1]\times I_{n,j_{k+1}}^{(k+1)}\times\cdots\times I_{n,j_d}^{(d)},\]
	the right of the above inclusion can be covered by
	\begin{equation}\label{eq:Bni(R)times}
		 \#  \cb_{i,n}(R)\cdot \frac{1}{6\psi_i(n)|\beta_i|^{-n}}
	\end{equation}
	balls of diameter $ 6\psi_i(n)|\beta_i|^{-n} $. Let \[\begin{split}
		\ck_{n,1}(i):&=\{1\le k\le d:|\beta_k|^{-n}<6\psi_i(n)|\beta_i|^{-n}\}\\
		&=\bigg\{1\le k\le d:\log|\beta_k|>-\frac{\log6+\log\psi_i(n)}{n}+\log|\beta_i|\bigg\},
	\end{split}\]
	and
	\[\begin{split}
		\ck_{n,2}(i):&=\{1\le k\le d:\psi_k(n)|\beta_k|^{-n}\ge \psi_i(n)|\beta_i|^{-n}\}\\
		&=\bigg\{1\le k\le d:-\frac{\log\psi_k(n)}{n}+\log|\beta_k|\le-\frac{\log\psi_i(n)}{n}+\log|\beta_i|\bigg\}.
	\end{split}\]
	Note that by \eqref{eq:Nn upper bound}, for any $ \epsilon>0 $,
	\[N_{k,n}\le |\beta_k|^{n(1+\epsilon)}\]
	for $ n $ sufficiently large.
	Then, in view of \eqref{eq:Bni(R) upper bound} and \eqref{eq:Bni(R)times}, there is a collection $ \cb_{i,n} $ of balls that covers the set $ R_n $ with
	\[\begin{split}
			 \#  \cb_{i,n}&\le \prod_{k\notin \ck_{n,1}(i)}|\beta_k|^{n(1+\epsilon)}\ \cdot\prod_{k\in \ck_{n,1}(i)}\frac{1}{6\psi_i(n)|\beta_i|^{-n}}\ \cdot\prod_{k\in\ck_{n,2}(i)}\frac{\psi_k(n)|\beta_k|^{-n}}{\psi_i(n)|\beta_i|^{-n}}\\
			 &= \prod_{k\notin \ck_{n,1}(i)}|\beta_k|^{n(1+\epsilon)}\ \cdot\prod_{k\in \ck_{n,1}(i)}\frac{|\beta_k|^n\cdot|\beta_k|^{-n}}{6\psi_i(n)|\beta_i|^{-n}}\ \cdot\prod_{k\in\ck_{n,2}(i)}\frac{\psi_k(n)|\beta_k|^{-n}}{\psi_i(n)|\beta_i|^{-n}}\\
			&\le\prod_{k=1}^d|\beta_k|^{n(1+\epsilon)}\cdot\prod_{k\in\ck_{n,1}(i)}\frac{|\beta_k|^{-n}}{\psi_i(n)|\beta_i|^{-n}}\cdot\prod_{k\in\ck_{n,2}(i)}\frac{\psi_k(n)|\beta_k|^{-n}}{\psi_i(n)|\beta_i|^{-n}}.
		\end{split}\]
	Note that the estimation of $ \# \cb_{i,n} $ is nearly identical to that of shrinking target problems \cite[Page 43]{LLVZ22}, with the only difference being the multiplicative term $ \prod_{k=1}^{d}|\beta_k|^{n\epsilon} $. However, this term does not affect the upper estimation of $ \hdim\rc(\Psi) $ because $ \epsilon $ can be made arbitrarily small. Thus, by utilizing an argument described in \cite[Pages 43--46]{LLVZ22}, we are able to establish the following Proposition \ref{p:dimension upper bound} on the upper bound for $ \hdim\rc(\Psi) $.

\begin{prop}\label{p:dimension upper bound}
	Under the setting of Theorem \ref{t:dimension}, we have
	\[\hdim\rc(\Psi)\le\sup_{\bt\in\cu(\Psi)}\min_{1\le i\le d}\theta_i(\bt).\]
\end{prop}
\subsection{Lower bound for $ \hdim\rc(\Psi) $}
The proof of the lower bound of $ \hdim\rc(\Psi) $ relies on a recent result of Wang and Wu \cite[Theorem 3.4]{WaWu21}, called the mass transference principle from rectangles to rectangles. For our purpose, we borrow a simplified version of Wang and Wu's result from \cite[Theorem 11]{LLVZ22}.

Let $ \nu $ be a probability measure supported on a locally compact set $ X\subset\R $. We say that $ \nu $ is $ \delta $-Ahlfors regular if there exist constants $ \delta>0 $, $ 0<a\le 1\le b<\infty $ and $ r_0>0 $ such that for any $ x\in X $ and $ r<r_0 $,
\begin{equation}\label{eq:Ahlfors regular}
	ar^\delta\le \nu(B(x,r))\le br^\delta,
\end{equation}
where $ B(x,r) $ is the ball with center $ x\in X $ and radius $ r $.
\begin{thm}[{\cite[Theorem 3.4]{WaWu21}}]\label{t:MTP RtoR}
	For each $ 1\le i\le d $, let $ X_i $ be a locally compact subset of $ \R $ equipped with a $ \delta_i $-Ahlfors regular measure $ \nu_i $. Let $ \{B(x_{i,n},r_{i,n})\}_{n\in\N} $ be a sequence of balls in $ X_i $ with radius $ r_{i,n}\to 0 $ as $ n\to\infty $ for each $ 1\le i\le d $ and assume that there exist $ \bv=(v_1,\dots,v_d)\in(\R_{\ge 0})^d $ and a sequence $ \{r_n\}_{n\in\N} $ of positive real numbers such that
	\[r_{i,n}=r_n^{v_i} \qquad\text{for all $ 1\le i\le d $}.\]
   Suppose that there exists $ (s_1,\dots,s_d)\in \prod_{i=1}^{d}(0,\delta_i) $ such that
	\begin{equation}\label{eq:enlarged ball full measure}
		\nu_1\times\cdots\times \nu_d\bigg(\limsup_{n\to\infty}\prod_{i=1}^{d}B\big(x_{i,n}, r_{i,n}^{s_i/\delta_i}\big)\bigg)=\nu_1\times\cdots\times \nu_d\bigg(\prod_{i=1}^{d}X_i\bigg).
	\end{equation}
	Then, we have
	\[\hdim \bigg(\limsup_{n\to\infty}\prod_{i=1}^dB(x_{i,n},r_{i,n})\bigg)\ge \min_{1\le i\le d}s(\bu,\bv,i),\]
	 where $ \bu=(u_1,\dots,u_d) $ with $ u_i=s_iv_i/\delta_i $ for $ 1\le i\le d $, and
	\[s(\bu,\bv,i):=\sum_{i\in\ck_1(i)}\delta_k+\sum_{i\in\ck_2(i)}\delta_k\bigg(1-\frac{v_k-u_k}{v_i}\bigg)+\sum_{i\in\ck_3(i)}\frac{\delta_ku_k}{v_i},\]
	with the sets
	\[\ck_1(i):=\{1\le k\le d:u_k\ge v_i\},\quad \ck_2(i):=\{1\le k\le d:v_k\le v_i\},\]
	and
	\[\ck_3(i):=\{1,\dots,d\}\setminus (\ck_1(i)\cup\ck_2(i))\]
	forming a partition of $ \{1,\dots,d\} $.
\end{thm}

The construction of the desired Cantor subset will make use of a general statement concerning Markov subsystems that has been explored in \cite[\S 4.3]{LLVZ22}.
Let $ X $ be a locally compact set in $ \R $ and $ f\colon X\to X $ be an expanding map. Let $ \Lambda $ be a subset of $ X $. A partition $ \cp_\Lambda $ of $ \Lambda $ into finite or countable collection of sets $ P(k) $ is called a {\em Markov partition} if $ \Lambda=\cap_{n=1}^\infty f^{-n}(\cup P(k)) $ and
\begin{enumerate}[(i)]
	\item if $ j\ne k $, then the interior of $ P(j) $ and $ P(k) $ are disjoint,
	\item $ f $ restricted on each $ P(j) $ is one to one,
	\item if $ f(P(j)) $ intersects the interior of $ P(k) $ for some $ j $ and $ k $, then $ P(k)\subset \overline{f(P(j))} $.
\end{enumerate}
The system $ (\Lambda, f|_\Lambda,\cp_\Lambda) $ is called a Markov subsystem of $ (X,f) $.


Here and below, when $ f $ is a piecewise linear function having a constant slope, we denote by $ \beta(f) $ the absolute value of such slope.

\begin{prop}[{\cite[Proposition 7]{LLVZ22}}]\label{p:Markov subsystem}
	Let $ f $ be a piecewise linear function on $ [0,1] $ having a constant slope. Assume that $ \beta(f)>8 $. Then there exists a Markov subsystem $ (\Lambda, f|_{\Lambda},\cp_\Lambda) $ of $ ([0,1],f) $ with a finite partition $ \cp_\Lambda=\{P(i)\} $ where each $ P(i) $ is an interval and $ f|_{P(i)} $ is linear, such that
	\[\hdim \Lambda\ge 1-\frac{\log 8}{\log\beta(f)}.\]
\end{prop}

As an application to  $\beta$-transformation, we have the following proposition.

\begin{prop}\label{p:full shift}
	Let $\beta\in\R$ such that $ |\beta|>1 $ and let $ T_\beta $ be the corresponding $ \beta $-transformation. For any $ \epsilon>0 $, there exist an integer $ k=k(\epsilon) $ and a full subshift $ (\Lambda_k=\Lambda_k(\epsilon),T_\beta^k|_{\Lambda_k}) $ of $ ([0,1], T_\beta^k) $ with a finite partition such that
	\[\hdim \Lambda_k\ge 1-\epsilon.\]

\end{prop}

\begin{proof}
	For any $ \epsilon>0 $, let $ m $ be an integer such that
	\begin{equation}\label{eq:condition m}
		m>\frac{1+\log 8}{\epsilon\log|\beta|}.
	\end{equation}	Let $ (\Lambda,T_\beta^m|_{\Lambda}, \cp_{\Lambda}) $ be a Markov subsystem of $ ([0,1], T_\beta^m) $ coming from Proposition \ref{p:Markov subsystem}. Then,
	\begin{equation}\label{eq:lower bound entropy}
		h_{\mathrm{top}}(T_\beta^m|_\Lambda)=\hdim\Lambda\cdot \log\beta(T_\beta^m|_\Lambda)\ge m\log|\beta|-\log 8,
	\end{equation}
	since $ \beta(T_\beta^m|_\Lambda)=|\beta|^m $.
	In \cite[Page 55]{LLVZ22}, the authors further proved that such a Markov subsystem is topological mixing, which is equivalent to saying that the corresponding incidence matrix $ A=(a_{ij})_{1\le i,j\le\# \cp_\Lambda} $ is primitive (see \cite[Exercise 10.2.2]{ViOl16}).

		Denote by $ (\Sigma_A^\N, \sigma) $ with $\Sigma_A^\N\subset \{1,2,\dots,\#\cp_\Lambda\}^\N $ the corresponding symbolic space of the dynamics of $ T_\beta^m|_\Lambda $ induced by $ A $ and $ \Sigma_A^n $ the set of words of length $ n $ in $ \Sigma_A^\N $.
		By Bowen's definition of topological entropy (see \cite{Bowen73}),
		\[h_{\mathrm{top}}(T_\beta^m|_\Lambda)=\lim_{n\to \infty}\frac{1}{n}\log \#\Sigma_A^n.\]
		 With this definition in mind, we see that for any $ n $ large enough, there exist $ 1\le i_0, i_{n-1}\le \#\cp_\Lambda $ such that
		 \[\begin{split}
		 	\#\{\omega_0\omega_1\cdots \omega_{n-1}\in\Sigma_A^n:\omega_0=i_0, \omega_{n-1}=i_{n-1}\}&\ge \frac{1}{(\#\cp_\Lambda)^2}\cdot \#\Sigma_A^n\\
		 	&\ge \frac{1}{(\#\cp_\Lambda)^2}\cdot e^{n(h_{\mathrm{top}}(T_\beta^m|_\Lambda)-1/4)}\\
		 	&\ge e^{n(h_{\mathrm{top}}(T_\beta^m|_\Lambda)-1/2)}.
		 \end{split}\]
		 Since the incidence matrix $ A $ is primitive, we can find a word $ (i_{n-1}i_1i_2\cdots i_{\ell-2}i_{0}) $ which belongs to $ \Sigma_A^\ell $ for some integer $ \ell\ge 2 $. Let
		 \[\Sigma'=\{\omega_0\cdots \omega_{n+\ell-2}\in\Sigma_A^{n+\ell-2}:\omega_0=i_0, \omega_{n-1}=i_{n-1},\omega_n=i_1,\dots, \omega_{n+\ell-3}=i_{\ell-2}\}.\]
		 That is, the word in $ \Sigma' $ begins with $ i_0 $ and ends with $ i_{n-1}i_1i_2\cdots i_{\ell-2} $.
		 By the Markov property of $ \Sigma_A^\N $, it is not difficult to verify that the concatenation of any two words in $ \Sigma' $ belongs to $ \Sigma_A^{2(n+\ell-2)} $. It therefore follows that $ (\Sigma')^\N\subset \Sigma_A^\N $ and the symbolic space $ \big((\Sigma')^\N, \sigma^{n+\ell-2}\big) $ is a full shift. Moreover, by increasing the integer $ n $ if necessary, we have
		 \[\begin{split}
		 	\#\Sigma'&=\#\{\omega_0\omega_1\cdots \omega_{n-1}\in\Sigma_A^n:\omega_0=i_0, \omega_{n-1}=i_{n-1}\}\\
		 	&\ge e^{n(h_{\mathrm{top}}(T_\beta^m|_\Lambda)-1/2)}\\
		 	&\ge e^{(n+\ell-2)(h_{\mathrm{top}}(T_\beta^m|_\Lambda)-1)}.
		 \end{split}\]
	 	Hence,
	 	\[h_{\mathrm{top}}(\sigma^{n+\ell-2}|_{(\Sigma')^\N})\ge e^{(n+\ell-2)(h_{\mathrm{top}}(T_\beta^m|_\Lambda)-1)}.\]

		 Now, let $ k=m(n+\ell-2) $ and consider the projection $ \pi $ from $ (\Sigma')^\N $ to a compact subset $ \Lambda_k $ of $ \Lambda $ given by
		 \[\omega=(\omega_i)_{i\ge 0}\in(\Sigma')^\N\quad\mapsto\quad \pi(\omega)=\bigcap_{i=0}^\infty\, (T_\beta^m)^{-i}(P(\omega_i)).\]
		  Then, the dynamical system $ (\Lambda_k, T_\beta^k|_{\Lambda_k}) $ is a full subshift with a finite partition and the Hausdorff dimension of $ \Lambda_k $ is at least
		  \[\begin{split}
		  	\hdim \Lambda_k&=\frac{h_{\mathrm{top}}(T_\beta^k|_{\Lambda_k})}{\log \beta(T_\beta^k)}=\frac{h_{\mathrm{top}}(\sigma^{n+j-2}|_{(\Sigma')^\N})}{\log \beta(T_\beta^k)}\\
		  	&\ge \frac{(n+j-2)(h_{\mathrm{top}}(T_\beta^m|_\Lambda)-1)}{k\log|\beta|}\\
		  	&= \frac{h_{\mathrm{top}}(T_\beta^m|_\Lambda)-1}{m\log|\beta|}\ge 1-\frac{1+\log 8}{m\log|\beta|}\\
		  	&\ge 1-\epsilon,
		  \end{split}\]
	  	where the second to the last inequality follows from \eqref{eq:lower bound entropy}, and the last inequality follows from \eqref{eq:condition m}.
\end{proof}
Since $ T_{\beta}^k|_{\Lambda_{k}} $ is a full subshift and is a piecewise linear function with slope $ \beta^k $, $ \Lambda_k $ can be regarded as a homogeneous self-similar set. Therefore, $ \Lambda_k $ supports a $ \delta $-Ahlfors regular measure with $ \delta=\hdim\Lambda_k $. Following some ideas from \cite[\S 3.1]{CWW19}, one can show that for any sufficiently small $ r>0 $ and for all $ n\ge 1 $,
\begin{equation}\label{eq:rsup}
	\{x\in \Lambda_k:|T_\beta^{kn} x-x|<r\}\supset \bigcup_{j=1}^{M_{n}}B(x_{n,j}, |\beta|^{-kn}r)\cap \Lambda_k
\end{equation}
and
\begin{equation}\label{eq:union lamk}
	\Lambda_{k}=\bigcup_{j=1}^{M_n}B(x_{n,j}, |\beta|^{-kn})\cap\Lambda_k,
\end{equation}
where $ x_{n,j} $ is a point in $ \Lambda_k $ and $ M_n $ is the number of such points.

The lemma below is needed in constructing the desired Cantor subset of $ \rc(\Psi) $. We omit its proof since the argument is similar to that of \cite[Lemma 9]{LLVZ22}.
\begin{lem}\label{l:subsequence}
	For $ 1\le i\le d $, let $ \psi_i:\R_{\ge 0}\to\R_{\ge 0} $ be a positive decreasing function. Then, for any positive integer $ k $, the set of accumulation points of the sequence
	\[\bigg\{\bigg(-\frac{\log\psi_1(kn)}{kn},\dots,-\frac{\log\psi_d(kn)}{kn}\bigg)\bigg\}_{n\ge 1}\]
	is equal to $ \cu(\Psi) $.
\end{lem}
The lower bound for the Hausdorff dimension of $ \rc(\Psi) $ can be proved by using the same method as in \cite[Proposition 5]{LLVZ22}, we include it for completeness.
\begin{prop}
		Under the setting of Theorem \ref{t:dimension}, we have
	\[\hdim\rc(\Psi)\ge\sup_{\bt\in\cu(\Psi)}\min_{1\le i\le d}\theta_i(\bt).\]
\end{prop}
\begin{proof}

	Given $ \epsilon>0 $. For $ 1\le i\le d $, let
	\[\left(\Lambda_{k}^{(i)}=\Lambda_k^{(i)}(\epsilon),\, T_{\beta_i}^k|_{\Lambda_{k}^{(i)}}\right)\]
	 be the full subshift arising from Proposition \ref{p:full shift} satisfying
	\[\hdim \Lambda_{k}^{(i)}\ge 1-\epsilon.\]
	By \eqref{eq:rsup}, we can construct a $ \limsup $ type subset of $ \rc(\Psi) $ as follows
	\begin{equation}\label{eq:Cantor subset}
		\rc(\Psi)\supset\limsup_{n\to\infty}\bigcup_{j_1=1}^{M_{1,n}}\cdots \bigcup_{j_d=1}^{M_{d,n}}\prod_{i=1}^{d}B\big(x_{n,j_i}^{(i)}, |\beta_i|^{-kn}\psi_i(kn)\big)\cap \Lambda_k^{(i)},
	\end{equation}
	where $ x_{n,j_i}^{(i)} $ is a point in $ \Lambda_{k}^{(i)} $ and $ M_{i,n} $ is the number of such points. Further, by \eqref{eq:union lamk}, we have for any $ n\ge 1 $,
	\begin{equation}\label{eq:cover whole space}
		\Lambda_{k}^{(i)}=\bigcup_{j_i=1}^{M_{i,n}}B\big(x_{n,j_i}^{(i)}, |\beta_i|^{-kn}\big)\cap\Lambda_k^{(i)}.
	\end{equation}

	 Now fix a point $ \bt=(t_1,\dots, t_d)\in\cu(\Psi) $. By Lemma \ref{l:subsequence} and the fact that $ \cu(\Psi) $ is bounded, there exists a subsequence $ (n_\ell)_{\ell\ge 1} $ such that
	\[\lim_{\ell\to\infty}\frac{-\log \psi_i(kn_\ell)}{kn_\ell}=t_i\quad\text{for all $ 1\le i\le d $}.\]
	For any $ 0<\epsilon<1 $, there exists $ N=N(\epsilon)>0 $ such that
	\begin{equation}\label{eq: log|beta|<}
		\frac{(1-\epsilon)\log|\beta_i|}{(1-\epsilon)\log|\beta_i|+t_i}\le\frac{\log|\beta_i|}{\log|\beta_i|+\frac{-\log\psi_i(kn_\ell)}{kn_\ell}}
	\end{equation}
	for all $ \ell\ge N $ and $ 1\le i\le d $. For $ 1\le i\le d $, let $ \delta_i=\delta_i(\epsilon)=\hdim \Lambda_{k}^{(i)}(\epsilon) $ and
	\[s_i:=\delta_i\cdot\frac{(1-\epsilon)\log|\beta_i|}{(1-\epsilon)\log|\beta_i|+t_i}.\]
	Then, we can re-write \eqref{eq: log|beta|<} as
	\[\big(|\beta_i|^{-kn_\ell}\psi_i(kn_\ell)\big)^{s_i/\delta_i}\ge |\beta_i|^{-kn_\ell}.\]
	This together with \eqref{eq:cover whole space} implies that for any $ \ell\ge N $
	\[\Lambda_{k}^{(i)}=\bigcup_{j_i=1}^{M_{i,n_\ell}}B\Big(x_{n_\ell,j_i}^{(i)},\big(|\beta_i|^{-kn_\ell}\psi_i(kn_\ell)\big)^{s_i/\delta_i} \Big)\cap\Lambda_k^{(i)},\]
	and so
	\[\begin{split}
		\prod_{i=1}^d\Lambda_{k}^{(i)}		=&\limsup_{n\to\infty}\bigcup_{j_1=1}^{M_{1,n}}\cdots \bigcup_{j_d=1}^{M_{d,n}}\prod_{i=1}^dB\Big(x_{n,j_i}^{(i)},\big(|\beta_i|^{-kn}\psi_i(kn)\big)^{s_i/\delta_i} \Big)\cap\Lambda_k^{(i)}.
	\end{split}\]
	In other words, the $ \limsup $ set of `$ (s_1,\dots,s_d) $-scaled up' rectangles appearing on the right of \eqref{eq:cover whole space} satisfies \eqref{eq:enlarged ball full measure} with $ X_i=\Lambda_{k}^{(i)} $ for each $ 1\le i\le d $. Applying Theorem \ref{t:MTP RtoR} with $ v_i=(1-\epsilon)\log|\beta_i|+t_i $ and $ u_i=(1-\epsilon)\log|\beta_i| $ $ (1\le i\le d) $, we obtain the lower bound
	\[\hdim \rc(\Psi)\ge \min_{1\le i\le d}s(i,\epsilon)\]
	where
	\[s(i,\epsilon):=\sum_{k\in \ck_1(i,\epsilon)}\delta_k+\sum_{k\in\ck_2(i,\epsilon)}\delta_k\bigg(1-\frac{t_k}{(1-\epsilon)\log|\beta_i|+t_i}\bigg)+\sum_{k\in\ck_3(i,\epsilon)}\delta_k \frac{(1-\epsilon)\log|\beta_k|}{(1-\epsilon)\log|\beta_i|+t_i}\]
	and where
	\[\begin{split}
		\ck_1(i,\epsilon):&=\{k:(1-\epsilon)\log|\beta_k|\ge(1-\epsilon)\log|\beta_i|+t_i\},\\
		\ck_2(i,\epsilon):&=\{k:(1-\epsilon)\log|\beta_k|+t_k\le (1-\epsilon)\log|\beta_i|+t_i\},\\
		\ck_3(i,\epsilon):&=\{1,\dots,d\}\setminus \big(\ck_1(i,\epsilon)\cup\ck_2(i,\epsilon)\big).\\
	\end{split}\]
	Fix $ 1\le i\le d $. Letting $ \epsilon\to 0 $, we have $ \ck_1(i,\epsilon)\to \{k:\log|\beta_k|\ge \log|\beta_i|+t_i\}=\ck_1(i) $, $ \ck_2(i,\epsilon)\to \ck_2(i) $, $ \ck_3(i,\epsilon)\to \ck_3(i) $ and
	\[\lim_{\epsilon\to 0}\delta_i=\lim_{\epsilon\to 0}\hdim\Lambda_k^{(i)}(\epsilon)=1.\]
	 Thus,
	\[\lim_{\epsilon\to 0} s(i,\epsilon)=\theta_i(\bt)\quad\text{and}\quad \hdim \rc(\Psi)\ge \min_{1\le i\le d}\theta_i(\bt).\]
	Since $ \bt\in\cu(\Psi) $ is arbitrary, we have
	\[\hdim\rc(\Psi)\ge \sup_{\bt\in\cu(\Psi)}\min_{1\le i\le d}\theta_i(\bt),\]
	which completes the proof.
\end{proof}
%
%
%

\end{document}